\documentclass[11pt,english]{smfart}

\usepackage{amsmath,leftidx,amsthm,sfheaders}
\usepackage[all]{xy}
\usepackage{xspace,smfthm}
\usepackage[psamsfonts]{amssymb}
\usepackage[latin1]{inputenc}
\usepackage{graphicx,color}
\usepackage{hyperref,fancyhdr}

\newtheorem*{remark}{\bf Remark}
\newtheorem*{question}{\bf Question}

\newtheorem{theorem}{\bf Theorem}[section]
\newtheorem{proposition}[theorem]{\bf Proposition}
\newtheorem{definition}[theorem]{\bf Definition}
\newtheorem{Theorem}{\bf Theorem}

\newtheorem*{claim}{\bf Claim}
\newtheorem{lemma}[theorem]{\bf Lemma}
\newtheorem{corollary}[theorem]{\bf Corollary}

\newtheorem*{fact}{\bf Fact}

\newtheorem{Exampl}[Theorem]{\bf Example}

\def\C{{\mathbb C}}

\def\R{{\mathbb R}}
\def\Z{{\mathbb Z}}

\def\Q{{\mathbb Q}}

\def\D{\mathbb{D}}

\def\J{\mathcal{J}}
\def\K{\mathcal{K}}

\def\Per{\textup{Per}}
\def\bif{\textup{bif}}

\def\Mand{\mathbf{M}}
\def\J{\mathcal{J}}

\def\un{\underline{n}}
\def\um{\underline{m}}

\def\and{{\quad\text{and}\quad}}

\title{Distribution of postcritically finite polynomials III: Combinatorial continuity}
\author{Thomas Gauthier \& Gabriel Vigny}
\address{LAMFA, Universit\'e de Picardie Jules Verne, 33 rue Saint-Leu, 80039 AMIENS Cedex 1, FRANCE}
\email{thomas.gauthier@u-picardie.fr}
\email{gabriel.vigny@u-picardie.fr}
\begin{document}

\begin{abstract}
In the first part of the present paper, we continue our study of the distribution of postcritically finite parameters in the moduli space of polynomials: we show the equidistribution of Misiurewicz and parabolic parameters with prescribed combinatorics toward the bifurcation measure. Our results essentially rely on a combinatorial description of the escape locus and of the bifurcation measure developped by Kiwi and Dujardin-Favre.
\par In the second part of the paper, we construct a bifurcation measure for the connectedness locus of the quadratic anti-holomorphic family which is supported by a strict subset of the boundary of the Tricorn. We also establish an approximation property by Misiurewicz parameters in the spirit of the previous one. Finally, we answer a question of Kiwi, exhibiting in the moduli space of degree $4$ polynomials, non-trivial Impression of specific combinatorics.
\end{abstract}

\maketitle
\tableofcontents

\section*{Introduction}
In this article, we study equidistribution problems in parameter spaces of polynomials. In any holomorphic family of rational maps, DeMarco~\cite{DeMarco1} introduced a current $T_\bif$ which is supported exactly on the bifurcation locus, giving a measurable point of view to study bifurcations. Bassanelli and Berteloot~\cite{BB1} considered the self-intersections of this current which enable to study higher bifurcations phenomena. In the moduli space $\mathcal{P}_d$ of degree $d$ polynomials, the maximal self-intersection of the bifurcation current induces a \emph{bifurcation measure}, $\mu_\bif$, which is the analogue of the harmonic measure of the Mandelbrot set when $d\geq 3$. It measures the sets of maximal bifurcation phenomena (see~\cite{favredujardin}).

In particular, we want to understand the distribution of the \emph{Misiurewicz parameters} (the parameters for which all the critical points are strictly preperiodic) and parabolic parameters (the parameters for which they are exactly $d-1$ neutral cycles). Such parameters play a central role in complex dynamics. They allow computations of Hausdorff dimension of parametric fractal sets (\cite{Shishikura2,Article1}). The Misiurewicz parameters also are special from the arithmetic point of view, since they are points of small height for a well-chosen Weil height (\cite{Ingram,favregauthier}).
 They also play a central role in the geometry of the bifurcation locus: it is in the neighborhood of such parameters that we can exhibit the most complicated geometric phenomena (for example small copies of the Mandelbrot set \cite{McMullen3, gauthier_indiana} or similarity between the Julia set and the bifurcation locus \cite{orsay2,TanLei2}).
 
 ~ 

 In $\mathcal{P}_d$, Dujardin and Favre proved in \cite{favredujardin} the density of Misiurewicz parameters in the support of the bifurcation measure (see \cite{buffepstein}  for the case of rational maps, which relies on the results of \cite{BB1}). Our goal here is to give a \emph{quantitative} version of this statement in order to have a better understanding  of the distribution of these parameters. This can be seen as a parametric version of Birkhoff's ergodic theorem and also as an arithmetic equidistribution statement. Moreover, many different approaches exist to study that question, making it a deep and rich subject.

Indeed, such results have already been achieved using pluripotential theoretic tools. Levin~\cite{Levin1} showed the equidistribution of PCF (postcritically finite) parameters toward the bifurcation measure in the quadratic family using extremal properties of the Green function of the Mandelbrot set. That approach has been extended by Dujardin and Favre to prove the equidistribution of maps having a preperiodic marked critical point toward the bifurcation current of that given critical point \cite{favredujardin} (see also \cite{okuyama:distrib} for a simplified proof in the hyperbolic case). Still relying on pluripotential theory, the authors proved the equidistribution of hyberbolic PCF parameters with exponential speed of convergence in \cite{distribGV}. Notice also the results of \cite{Dujardin2012} for the case of intermediate bidegree. 

Another fruitful approach was the use of arithmetic methods, using the theorem of equidistribution of points of small height. Indeed, PCF parameters are arithmetic. Favre and Rivera-Letelier first used that approach in the quadratic family to prove the equidistribution of PCF parameters toward the bifurcation measure, with an exponential speed of convergence. That result was extended to the case of $\mathcal{P}_d$ ($d\geq 3$) by Favre and the first author in \cite{favregauthier}. Notice that the equidistribution statement proved in~\cite{favregauthier} requires technical assumptions on the pre-periods and periods of the critical orbits of the considered parameters. 

~

In here, we develop instead a combinatorics approach, based on the impression of external rays. The first part of this article is dedicated to the proof of Theorem~\ref{tm:distrib} which is of different nature. Indeed, we impose conditions on the combinatorics of angles with given period and preperiod landing at critical points instead of giving conditions for the parameter itself (see Section~\ref{sec:combinPd}). On the other hand, we make no assumption on the periods and preperiods of the critical orbits. For that, we develop further the arguments of Dujardin and Favre \cite{favredujardin}, using Kiwi's results on the combinatorial space and the landing of external rays (\cite{kiwi-portrait, kiwireal}) and the results of Przytycki and Rohde~\cite{PR} on the rigidity of Topological Collet-Eckmann repellers.

Let $\mathsf{Cb}$ be the space of critical portraits of degree $d$ polynomials (see Section~\ref{sec:prelim} for a precise definition). Pick now any $(d-1)$-tuples $\un=(n_0,\ldots,n_{d-2})$ and $\um=(m_0,\ldots,m_{d-2})$ of non-negative integers with $m_i>n_i$. We let
\[\mathsf{C}(\um,\un):=\{(\Theta_0,\ldots,\Theta_{d-2})\in\mathsf{Cb}\, ; \ d^{m_i}\theta=d^{n_i}\theta~, \ \forall \theta\in\Theta_i, \ \forall i\}.\]
When $m_i>n_i\geq1$ for all $i$, we also let
\[\mathsf{C}^*(\um,\un):=\mathsf{C}(\um,\un)\setminus\mathsf{C}(\um-\un,\underline{0}).\]
A critical portrait $\Theta\in\mathsf{C}^*(\um,\un)$ is called \emph{Misiurevwicz}. The space $\mathsf{Cb}$ is known to admit a natural probability measure $\mu_{\mathsf{Cb}}$. Dujardin and Favre have also built a \emph{landing map} $e:\mathsf{Cb}\longrightarrow\mathcal{P}_d$ which satisfies $e_*(\mu_\mathsf{Cb})=\mu_\bif$ and which sends Misiurewicz combinatorics to Misiurewicz polynomials (see Sections~\ref{sec:combinPd} and~\ref{sec:bifmes} for more details).

Our first result can be stated as follows.

\begin{Theorem}\label{tm:distrib}
Let $(\un_k)_k$ and $(\um_k)_k$ be two sequences of $(d-1)$-tules with $m_{k,j}>n_{k,j}\geq1$ and $m_{k,j}\to\infty$ as $k\to\infty$ for all $j$.  Let $X_k:=e(\mathsf{C}^*(\um_k,\un_k))$ and let $\mu_k$ be the measure
\[\mu_k:=\frac{1}{\textup{Card}(\mathsf{C}^*(\um_k,\un_k))}\sum_{\{P\}\in X_k}\mathcal{N}_{\mathsf{Cb}}(P)\cdot\delta_{\{P\}}~,\]
where $\mathcal{N}_{\mathsf{Cb}}(P)$ is the (finite) number of distinct critical portraits of the polynomial $P$. Then $\mu_k$ converges to $\mu_\bif$ as $k\to\infty$ in the weak sense of probablility measures on $\mathcal{P}_d$.
\end{Theorem}

Notice that the support of $\mu_k$ is contained in the set of classes $\{P\}\in\mathcal{P}_d$ such that $P^{n_{k,j}}(c_j)=P^{m_{k,j}}(c_j)$ and $c_j$ is not periodic. Remark also that the above result does not deal with an \emph{equidistribution} property, since the considered measures take into account the \emph{combinatorial multiplicity} $\mathcal{N}_\mathsf{Cb}(P)$ of Misiureiwcz parameters. We give in Section~\ref{sec:combinPd} a description of the range of $\mathcal{N}_\mathsf{Cb}$.

Then, using a general version of the theory of Douady and Hubbard of landing of external rays for parabolic combinatorics, we also prove a similar result of equidistribution of parameters having a parabolic combinatorics toward the bifurcation measure. More precisely, we have the following result.

\begin{Theorem}\label{tm:distribpara}
Let $(\un_k)_k$ be any sequence of $(d-1)$-tuples with $n_{k,j}\to\infty$ as $k\to\infty$ for all $j$ and $\gcd(n_{k,j},n_{k,i})=1$ for all $k$ and $j\neq i$.  Let $Y_k:=e(\mathsf{C}(\un_{k},\underline{0}))$ and $\mu'_k$ be the measure
\[\mu'_k:=\frac{1}{\textup{Card}(\mathsf{C}(\un_{k},\underline{0})) }\sum_{\{P\}\in Y_k}\mathcal{M}_{\mathsf{Cb}}(P)\cdot\delta_{\{P\}}~,\]
where $\mathcal{M}_{\mathsf{Cb}}(P)$ is the (finite) number of combinatorics lying in $\mathsf{C}(\un_{k},\underline{0})$ whose impression is reduced to $\{P\}$ . Then $\mu'_k$ converges to $\mu_\bif$ as $k\to\infty$ in the weak sense of probability measures on $\mathcal{P}_d$.
\end{Theorem}

Notice that those parameters have $(d-1)$ distinct parabolic cycles. In order to obtain the equidistribution of totally parabolic polynomials, one would need a precise control on the cardinality of combinatorics that land at a given totally parabolic polynomial. It is nevertheless the only general result in that direction existing so far.

~

In the second part of the present work, we adapt the above combinatorial methods to the case of the parameters space of quadratic antiholomorphic polynomials, i.e. the family 
\[f_c(z):=\bar{z}^2+c~, \ c\in\C.\]
The connectedness locus, in that setting, is known as the Tricorn $\Mand_2^*$. As observed by Inou and Mukherjee in \cite{inoumuk}, the harmonic measure of the Tricorn is not a good candidate to measure bifurcation phenomena: the existence of (real analytic) stable parabolic arcs does not allow the density of PCF parameters.  We develop further the theory of landing map of external rays in that setting. Precisely, we prove the following.

\begin{Theorem}\label{tm:mubifanti}
Almost any external ray of the Tricorn $\Mand_2^*$ lands and, if $\ell:\R/\Z\longrightarrow\partial\Mand_2^*$ is the landing map, then $\ell$ is measurable and there exists a set $\mathsf{R}\subset\R/\Z$ of full Lebesgue measure such that $\ell|_\mathsf{R}$ is continuous.
\end{Theorem}

The use of external rays for the Tricorn has been initiated by Nakane in \cite{Nakane} to prove the connectedness of the Tricorn. A finer study of the topological and combinatorial properties was developped by several authors (e.g. \cite{HubbardSchleicher} where the authors showed that the Tricorn is not path connected).

To prove Theorem~\ref{tm:mubifanti}, we imbed the quadratic antiholomorphic polynomials family in a complex family of degree 4 polynomials maps in order to use again Kiwi's results on the combinatorial space to prove the equidistribution of Misiurewicz parameters. We now define the \emph{bifurcation measure} of the Tricorn $\Mand_2^*$ as
\[\mu_\bif^*:=(\ell)_*\left(\lambda_{\R/\Z}\right)~.\]
 We believe that this measure should equidistribute other dynamical phenomena, as hyperbolic postcritically finite parameters for example. For $n>k>0$, we consider the following set of \emph{Misiurewicz parameters}:
\[\Per^*(n,k):=\{c\in\C\, ; f_c^n(0)=f_c^k(0) \ \text{and} \ f^{n-k}_c(0)\neq0\}~,\]
similarly we consider the following set of \emph{Misiurewicz combinatorics}: 
\[\mathsf{C}^*(n,k):=\{\theta\in \R/\Z \, ; (-2)^{n-1}\theta=(-2)^{k-1}\theta\ \text{and} \ (-2)^{n-k}(\theta)\neq\theta\}~.\]

Building on the above definition of the bifurcation measure, we can describe the distribution of the sets $(\ell)_*(\mathsf{C}^*(n,k))$ which is a subset of $\Per^*(n,k)$. This is the content of our next result.
\begin{Theorem}\label{tm:distribanti}
For any $1<k<n$, the set $\Per^*(n,k)$ is finite and $(\ell)_*(\mathsf{C}^*(n,k)) \subset Per^*(2n,2k)$. Moreover, for any sequence $1<k(n)<n$, the measure 
\[\mu_n^*:= \frac{1}{\textup{Card}(\mathsf{C}^*(n,k(n)))} \sum_{c\in(\ell)_*(\mathsf{C}^*(n,k(n)))}\mathcal{N}_{\R/\Z}(c)\cdot \delta_c~,\]
where $\mathcal{N}_{\R/\Z}(c):=\textup{Card}\{\theta\in\mathsf{C}^*(n,k(n))\, ; \ \ell(\theta)=c\}\geq1$, converge to $\mu_\bif^*$ in the weak sense of measures on $\C$. 
\end{Theorem}

In a certain sense, parameters of $(\ell)_*(\mathsf{C}^*(n,k(n)))$ are \emph{truly} of pure period $n-k(n)$ and preperiod $k(n)$, since their combinatorics also have the same property.

Notice that the question of counting parameters such that $f_c^n(0)=f_c^k(0)$ is of \emph{real} algebraic nature and is difficult. On that matter, notice the difficult work \cite{MNS} where the authors notably count the number of hyperbolic components of the Tricorn. 

In order to prove the above result, we relate the Misiurewicz character of $f_c$ to the Misiurewicz character of the induced degree $4$ polynomials for which we can apply known results in landing of external rays. We also relate the measure $\mu_\bif^*$ to the bifurcation measure $\nu_\bif$ of that family of degree $4$ polynomials by the inclusion $\mathrm{supp}(\mu_\bif^*)\subset \mathrm{supp}(\nu_\bif)\cap \mathbb{R}^2$. This follows from the fact that Misiurewicz parameters belong to the support of $\nu_\bif$ and are dense in it. 

~

Let us now state our last result, coming back to the moduli space $\mathcal{P}_4$ of critically marked degree $4$ polynomials. Relying the main result of Inou and Mukherjee~\cite{inoumuk}, we exhibit the following family of examples, answering in the case $d=4$ a question asked by Kiwi in his PhD Thesis (see \cite[\S14, page 42]{kiwithese}).

\begin{Exampl}[Non-trivial Impressions]\label{tm:non-trivial}
There exists an infinite set of critical portraits $\Theta\in\mathsf{Cb}$ for which the impression $\mathcal{I}_{\mathcal{C}_4}(\Theta)$ of $\Theta$ in the moduli space of degree $4$ is not reduced to a point. More precisely, the impression of such $\Theta$ can be chosen to contain a non-trivial smooth arc consisting in polynomials $P$ having a parabolic cycle of given period attracting all critical points of $P$.
\end{Exampl}

In a first section, we start with general preliminaries, notably on the combinatorial space and the landing of external rays. We then give the proof of Theorem~\ref{tm:distrib} in the particular case of the quadratic family. That proof is of folklore nature in that case but we believe it will help the global understanding of the reader.
 In Part \ref{Part:modd}, we develop a general version of the theory of Douady and Hubbard of landing of external rays for parabolic combinatorics and we prove Theorem~\ref{tm:distrib} and Theorem~\ref{tm:distribpara}. In Part~\ref{Part:anti}, we treat the case of the Tricorn. We start by exploring the combinatorial space in that setting and deduce Example~\ref{tm:non-trivial} above. Finally, we prove Theorems~\ref{tm:mubifanti} and \ref{tm:distribanti}.

\paragraph*{Acknowledgement}
Both authors are partially supported by the ANR project Lambda ANR-13-BS01-0002.

\section{General preliminaries}\label{sec:prelim}

\subsection{The moduli space and the visible shift locus}

The \emph{moduli space} $\mathcal{P}_d$ of degree $d$ polynomials is the space of affine conjugacy classes of degree $d$ polynomials with $d-1$ marked critical points. A point in $\mathcal{P}_d$ is represented by a $d$-tuple $(P, c_0,\ldots, c_{d-2})$ where $P$ is a polynomial of degree $d$, and the $c_i$'s are complex numbers such that $\{c_0,\ldots,c_{d-2}\}$ is the set of all critical points of $P$. For each $i$, $\textup{Card}\{j,\ c_j = c_i\}$ is the order of vanishing of $P'$ at $c_i$. Two points $(P,c_0,\ldots,c_{d-2})$ and $(\tilde P,\tilde c_0,\ldots,\tilde c_{d-2})$ are identified when there exists an affine map $\phi$ such that $\tilde P=\phi\circ P\circ \phi^{-1}$, and $\tilde c_i = \phi(c_i)$ for all $0\leq i\leq d-2$.

The set $\mathcal{P}_d$ is a quasiprojective variety of dimension $d-1$, and is isomorphic to the quotient of $\C^{d-1}$ by the finite group of $(d-1)$-th roots of unity acting linearly and diagonally on $\C^{d-1}$ (see \cite{Silverman}). When $d\geq3$, this space admits a unique singularity at the point $(z^d, 0,\ldots, 0)$.

~

Recall that for $(P,c_0,\ldots,c_{d-2})\in\mathcal{P}_d$, the \emph{Green function of} $P$ is defined by
\[g_P(z):=\lim_{n\to+\infty}d^{-n}\log^+|P^n(z)|, \ z\in\C~.\]
It satisfies $\K_P=\{g_P=0\}$ and it is a psh and continuous function of $(P,z)\in\mathcal{P}_d\times\C$. Let
\[G(P):=\max_{0\leq j\leq d-2}g_P(c_j)~.\]
The \emph{connectedness locus} $\mathcal{C}_d:=\{(P,c_0,\ldots,c_{d-2})\in\mathcal{P}_d \, ; \ \J_P \text{ is connected}\}$ is a compact set and satisfies $\mathcal{C}_d=\{(P,c_0,\ldots,c_{d-2})\in\mathcal{P}_d \, ; \ G(P)=0\}$ (see \cite{BH}).

~

We also call \emph{B\"ottcher coordinate} of $P$ at infinity the unique biholomorphic map
\[\phi_P:\C\setminus\{z\in\C \, ; \ g_P(z)\leq G(P)\}\longrightarrow\C\setminus\overline{\D}(0,\exp(G(P)))\]
which is tangent to the identity at infinity and satisfies
\begin{enumerate}
\item $\phi_P\circ P=(\phi_P)^d$ on $\C\setminus\{z\in\C \, ; \ g_P(z)\leq G(P)\}$,
\item $g_P(z)=\log|\phi_P(z)|$ for all $z\in \C\setminus\{z\in\C \, ; \ g_P(z)\leq G(P)\}$.
\end{enumerate}
The \emph{external ray} for $P$ of angle $\theta$ is the set 
\[R_\theta:=\phi_P^{-1}(]\exp(G(P)),+\infty[e^{2i\pi\theta})~.\]
Such a ray may be extended as a smooth flow line of the gradient $\nabla g_P$ in $\{g_P>0\}=\C\setminus\K_P$. If it meets a critical point $c_i$ of $P$, we say that $R_\theta$ terminates at $c_i$. 

\begin{definition}
We say that a $(P,c_0,\ldots,c_{d-2})\in\mathcal{P}_d$ lies in the \emph{shift locus} $\mathcal{S}_d$ if all critical points of $P$ escape under iteration. We also say that a class $(P,c_0,\ldots,c_{d-2})\in\mathcal{S}_d$ lies in the \emph{visible shift locus} $\mathcal{S}_{d}^{\textup{vis}}$ if for all $0\leq i\leq d-1$, there exists $\deg_{c_i}(P)$ external rays terminating at $c_i$ and $P(c_i)$ belongs to an external ray.
\end{definition}

When $(P,c_0,\ldots,c_{d-2})\in\mathcal{S}^{\textup{vis}}_d$, we denote by $\Theta(P)$ the \emph{combinatorics} (or \emph{critical portrait}) of $P$, i.e. the $(d-1)$-tuple
$\Theta(P):=(\Theta_0,\ldots,\Theta_{d-2})$ of finite subsets of $\R/\Z$ for which $\Theta_i$ is exactly the collection of angles of rays landing at $c_i$.

\subsection{The combinatorial space}
We follow the definition given by Dujardin and Favre \cite{favredujardin}.  Two finite and disjoint subsets $\Theta_1,\Theta_2\subset\R/\Z$ are said to be \emph{unlinked} if $\Theta_2$ is included in a single connected component of $(\R/\Z) \setminus \Theta_1$.
We let $\mathsf{S}$ be the set of pairs $\{\alpha,\alpha'\}$ contained in the circle $\R/\Z$, such that $d\alpha=d\alpha'$ and $\alpha\neq\alpha'$.
 First, we can define the simple combinatorial space.
\begin{definition}
We let $\mathsf{Cb}_0$ be the set of $(d-1)$-tuples $\Theta=(\Theta_0,\ldots,\Theta_{d-2})\in\mathsf{S}^{d-1}$ such that for all $i\neq j$, the two pairs $\Theta_i$ and $\Theta_j$ are disjoint and unlinked.
\end{definition}

It is known that $\mathsf{Cb}_0$ has a natural structure of translation manifold. It is also known to carry a natural invariant probability measure that we will denote $\mu_{\mathsf{Cb}_0}$ (see \cite[\S 7]{favredujardin}).
We now may define the full combinatorial space.

\begin{definition}
The set $\mathsf{Cb}$ is the collection of all $(d-1)$-tuples $(\Theta_0,\ldots,\Theta_{d-2})$ of finite sets in $\R/\Z$ satisfying the following four conditions:
\begin{itemize}
\item for any fixed $i$, $\Theta_i=\{\theta_1,\ldots,\theta_{k(i)}\}$ and $d\theta_j = d\theta_1$ for all $j$;
\item for any $i\neq j$, either $\Theta_i=\Theta_j$ or $\Theta_i\cap\Theta_j=\emptyset$;
\item if $N$ is the total number of distinct $\Theta_i$'s, then $\textup{Card} \bigcup_i\Theta_i=d+N-1$;
\item for any $i\neq j$ such that $\Theta_i\cap\Theta_j=\emptyset$, the sets $\Theta_i$ and $\Theta_j$ are unlinked.
\end{itemize}
\end{definition}

\begin{remark}
When $(P,c_0,\ldots,c_{d-2})\in\mathcal{S}_d^{\textup{vis}}$, then $\Theta(P)\in\mathsf{Cb}$.
\end{remark}

We will use the following.

\begin{proposition}[Kiwi, Dujardin-Favre]\label{tm:muCb}
The set $\mathsf{Cb}$ is compact and path connected and contains $\mathsf{Cb}_0$ as a dense open subset.
\end{proposition}

Then, we define, as Dujardin and Favre, the combinatorial measure $\mu_{\mathsf{Cb}}$ as the only probability measure on $\mu_\mathsf{Cb}$ which coincides with $\mu_{\mathsf{Cb}_0}$ on $\mathsf{Cb}_0$ and does not charge $\mathsf{Cb}\setminus\mathsf{Cb}_0$.

Following Kiwi \cite{kiwi-portrait}, we will use the following definition.
\begin{definition}
Pick $\Theta\in\mathsf{Cb}$. We say that $P$ lies in the \emph{impression} of $\Theta$ if there exists a sequence $P_n\in\mathcal{S}_d^\textup{vis}$ converging to $P$ such that the corresponding critical portraits $\Theta(P_n)$ converge to $\Theta$.
\end{definition}

We denote by $\mathcal{I}_{\mathcal{C}_d}(\Theta)$ the impression of any $\Theta\in\mathsf{Cb}$. Kiwi proved the following result concerning basic properties of the impression of a combinatorics (see \cite{kiwi-portrait}).

\begin{proposition}\label{prop:kiwi}
For any $\Theta\in\mathsf{Cb}$, the impression $\mathcal{I}_{\mathcal{C}_d}(\Theta)$ is a non-empty connected subset of $\partial \mathcal{C}_d\cap \partial \mathcal{S}_d$.
\end{proposition}

According to Theorem~5.12 of \cite{kiwireal}, whenever $\J_P=\K_P$ is locally connected and $P$ has no irrationally neutral cycle, the map $P:\J_P\longrightarrow\J_P$ is conjugate to the maps induced by $z\mapsto z^d$ on a quotient $\mathbb{S}^1/\sim_P$ of $\mathbb{S}^1$ by a dynamically defined equivalence relation. Moreover, Theorem~1 of \cite{kiwi-portrait} guarantees that if $P,P'\in\mathcal{I}_{\mathcal{C}_d}(\Theta)$ have only repelling cycles and have locally connected Julia sets $\J_P=\K_P$ and $\J_{\tilde P}=\K_{\tilde P}$, the quotient spaces $\mathbb{S}^1/\sim_P$ and $\mathbb{S}^1/\sim_{P'}$ depend only on the combinatorics $\Theta$, and in particular are homeomorphic.

All this summarizes as follows.

\begin{theorem}[Kiwi]\label{tm:kiwi}
Let $\Theta\in\mathsf{Cb}_0$ and let $(P,c_0,\ldots,c_{d-2}),(\tilde P,\tilde c_0,\ldots,\tilde c_{d-2})\in\mathcal{I}_{\mathcal{C}_d}(\Theta)$. Assume that $\J_P=\K_P$ and $\J_{\tilde P}=\K_{\tilde P}$ are locally connected and that $P$ and $\tilde P$ have only repelling cycles. Then there exists an orientation preserving homeomorphism $h:\J_P\longrightarrow\J_{\tilde P}$ which conjugates $P$ to $\tilde P$ on their Julia sets.
\end{theorem}

In the sequel, we will use this result in the following way: when the impression $\mathcal{I}_{\mathcal{C}_d}(\Theta)$ of $\Theta$ contains a polynomial which is topologically rigid, with locally connected Julia set and having only repelling cycles, then $\mathcal{I}_{\mathcal{C}_d}(\Theta)$ is reduced to the singleton $\{P\}$. We will be particularly interested in the case where $P$ is  Topological Collet-Eckmann.

~

Recall that a polynomial $P$ satisfies the \emph{topological Collet-Eckmann} (or TCE) condition if for some $A\geq1$
 there exist constants $M>1$ and $r>0$ such that for every $x\in \J_P$ there is an increasing sequence $(n_j)$ with $n_j\leq A\cdot j$ such that for every $j$,
\[ \textup{Card}\left\{i\, ; \ 0\leq i<n_j, \textup{Comp}_{f^i(x)}f^{-(n_j-i)}\D(f^{n_j}(x),r)\cap C(P)\neq \emptyset\right\}\leq M,\]
where $\textup{Comp}_x(X)$ is the connected component of the set $X$ containing $x$.  It is known that if $P$ is TCE, then $\J_P$ 
is locally connected, $C(P)\subset \J_P=\K_P$ and $P$ only has repelling cycles (see e.g. \cite[Main Theorem]{PRLS}).

\subsection{Measure theoretic tools}\label{sec:measuretheory}
 A classical result states that if $f:X\longrightarrow Y$ is a map between metric spaces, $\nu$ is a probability measure on $X$ 
and $\nu_n$ converges weakly to $\nu$ and if the set $D_f$ of discontinuities of $f$ satisfies $\nu(D_f)=0$, then $f_*(\nu_n)$ 
converges weakly to $f_*(\nu)$. This is known as the \emph{mapping theorem}. We prove the following slight generalization we will use in a crucial way.

\begin{theorem}\label{tm:measure}
Let $(X,d_X)$ and $(Y,d_Y)$ be metric spaces and let $f:X\longrightarrow Y$ be a measurable map. Let $\nu$ be a probability
 measure on $X$ such that there exists a Borel subset $S\subset X$ with $\nu(S)=1$ and such that $f|_S:S\longrightarrow Y$ is continuous. 
Pick any sequence $\nu_n$ of probability measures on $X$ with $\nu_n(S)=1$. Assume in addition that $\nu_n$ converges weakly to $\nu$ on $X$. 
Then $f_*(\nu_n)$ converges weakly to $f_*(\nu)$.
\end{theorem}

We rely on the following classical fact (see e.g.~\cite[Theorem 2.1 p. 16]{billingsley}).

\begin{fact}
A sequence of probability measures $\mu_n$ on a metric space converges weakly to a probability measure $\mu$ if and only if $\limsup_{n\to\infty}\mu_n(F)\leq \mu(F)$ for any closed set $F$, or equivalently, if and only if $\liminf_{n\to\infty}\mu_n(U)\geq \mu(U)$ for any open set $U$.
\end{fact}

\begin{proof}[Proof of Theorem~\ref{tm:measure}]
First, $f_*(\nu)$ is a probability measure on $Y$ and its restriction to $W:=f(S)$ is a probability measure since $\nu(S)=1$. Let $g:=f|_S:S\to W$ and $\tilde\nu_n:=\nu_n|_S$ and $\tilde \nu:=\nu|_S$. Notice that, by assumption, $g$ is a continuous map between metric spaces and $\left.f_*(\nu_n)\right|_W=g_*(\tilde\nu_n)$ and $\left.f_*(\nu)\right|_W=g_*(\tilde\nu)$. Moreover, $\tilde\nu_n$ and $\tilde \nu$ are probability measures on the metric space $S$ and $\tilde \nu_n$ converges weakly to $\tilde\nu$ on $S$.

Pick now any closed subset $B\subset Y$ and let $B':=W\cap B$. We have that $\nu_n(f^{-1}(B))=\tilde\nu_n(g^{-1}(B'))$ and $\nu(f^{-1}(B))=\tilde\nu(g^{-1}(B'))$, since $\nu_n(X\setminus S)=\nu(X\setminus S)=0$. Recall also that $B'$ is a closed subset of $W$ and $g^{-1}(B')$ is a closed subset of $S$. Hence, by the above Fact,
\begin{eqnarray*}
\limsup_{n\to\infty}f_*(\nu_n)(B) & = & \limsup_{n\to\infty}g_*(\tilde\nu_n)(B')=\limsup_{n\to\infty}\tilde\nu_n(g^{-1}(B'))\\
& \leq & \tilde\nu\left(g^{-1}(B')\right) =  \nu\left(f^{-1}(B)\right)=f_*(\nu)(B)~.
\end{eqnarray*}
Again by the Fact, this ends the proof.
\end{proof}

We say that a sequence $(A_n)_{n\geq0}$ of finite subsets of $\R/\Z$ is equidistributed if $\lim_{n\rightarrow\infty}\textup{Card}(A_n)=+\infty$ 
and if for any open interval $I\subset\R/\Z$ we have
\[\lim_{n\rightarrow+\infty}\frac{\textup{Card}(A_n\cap I)}{\textup{Card}(A_n)}=\lambda_{\R/\Z}(I).\]
We shall also use the following easy lemma.

\begin{lemma}\label{lm:equidsitrcircle}
Pick $(A_n)_{n\geq0}$ and $(B_n)_{n\geq0}$ two sequences of finite sets of $\R/\Z$. Assume that $(A_n)$ and $(B_n)$ are equidistributed, that $B_n\subset A_n$ and that 
\[\liminf_{n\rightarrow\infty}\frac{\textup{Card}(A_n\setminus B_n)}{\textup{Card}(A_n)}>0.\]
Then the sequence $(A_n\setminus B_n)$ is equidistributed. 
\end{lemma}

\begin{proof}
From the above fact, it is sufficient to check that for any open interval $I\subset\R/\Z$,  we have $\liminf_{n\rightarrow\infty}\mu_n(I)\geq \lambda_{\R/\Z}(I)$, since any open subset of $\R/\Z$ is a disjoint union of open intervals. Pick an open interval $I\subset \R/\Z$ and $\epsilon>0$. As $A_n$ and $B_n$ are equidistributed, there exists $n_0\geq1$ such that for any $n\geq n_0$, 
\[\left|\frac{\textup{Card}(A_n\cap I)}{\textup{Card}(A_n)}-\lambda_{\R/\Z}(I)\right|\leq \epsilon \ \text{and} \ \left|\frac{\textup{Card}(B_n\cap I)}{\textup{Card}(B_n)}-\lambda_{\R/\Z}(I)\right|\leq \epsilon.\]
Let $\alpha:=\liminf_n\textup{Card}(A_n\setminus B_n)/\textup{Card}(A_n)$. By assumption, we have $\alpha>0$ and up to increasing $n_0$, for any $n\geq n_0$ we may assume $\textup{Card}(A_n\setminus B_n)/\textup{Card}(A_n)\geq\alpha/2>0$. Hence
\begin{eqnarray*}
\left|\frac{\textup{Card}((A_n\setminus B_n)\cap I)}{\textup{Card}(A_n\setminus B_n)}-\lambda_{\R/\Z}(I)\right| & \leq & \frac{4}{\alpha}\epsilon.
\end{eqnarray*}
This concludes the proof.
\end{proof}

As a direct consequence, we see that the probability measure $\mu_n$ equidistributed on $A_n\setminus B_n$ converges weakly towards $\lambda_{\R/\Z}$.

\section{In the quadratic family}
This section serves as a model to the sequel: we develop our strategy in the family
\[p_c(z):=z^2+c~, \ (c,z)\in\C^2\]
which parametrizes the moduli space of quadratic polynomials.

In the present section, we prove a continuity property for the Riemann map of the complement of the Mandelbrot set and deduce Theorems~\ref{tm:distrib} and~\ref{tm:distribpara} in the present context from it and from known landing properties of rational angles (see e.g.~\cite{orsay1,orsay2,Schleicher}).

\subsection{Prime-End Impressions and Collet-Eckmann parameters}

All the material for this section is classical (see e.g.~\cite{orsay1,orsay2}). Recall that the bifurcation measure of the quadratic family $(p_c)_{c\in\C}$ is the harmonic measure $\mu_\Mand$ of the Mandelbrot set $\Mand:=\{c\in\C \, ; \ |p_c^n(0)|\leq2~, \forall n\geq0\}=\{c\in\C\, ; \ \J_c \text{ is connected}\}$. Moreover, the map
\[\Phi:\C\setminus\Mand\longrightarrow\C\setminus\overline{\D}\]
defined by $\Phi(c):=\phi_{p_c}(c)$ is a biholomorphism which is tangent to the identity at $\infty$. The \emph{external ray} of the Mandelbrot set of angle $\theta\in\R/\Z$ is the set
\[\mathcal{R}_\Mand(\theta):=\Phi^{-1}\left(\{Re^{2i\pi\theta}\, ; \ R>1\}\right).\]
The combinatorial space $\mathsf{Cb}$ is then $\mathsf{Cb}=\{\{\alpha,\alpha+\frac{1}{2}\}\, ; \ \alpha\in\R/\Z\}$. The impression of the combinatorics $\Theta=\{\alpha,\alpha+\frac{1}{2}\}$ can be described as the prime-end impression of the ray $\theta=2\alpha$ under the map $\Phi^{-1}$, i.e. as the set 
\[\mathcal{I}_{\Mand}(\Theta)=\bigcap_{\rho>1,\ \epsilon>0}\overline{\Phi^{-1}\left(\{Re^{2i\pi \tau}\, ; \ |\theta-\tau|<\epsilon, \ 1<R<\rho\}\right)}.\]
We say $\theta$ is \emph{Misiurewicz} if there exists $n>k\geq1$ such that $2^n\theta=2^k\theta$ and $2^{n-k}\theta\neq\theta$. Combining \cite[Lemma 4.1]{Schleicher} with \cite[Theorem 5.3]{kiwi-portrait}, we have the following.

\begin{proposition}\label{prop:kiwimand}
For any Misiurewicz angle $\theta$, the prime-end impression of $\theta$ is reduced to a singleton. Moreover, this singleton consists in a Misiurewicz parameter $c$. In particular, the ray $\mathcal{R}_\Mand(\theta)$ lands at $c$.

For any Misiurewicz parameter $c$, at least one ray lands at $c$ and the angles of the rays that land at $c$ are exactly the angles of the dynamical rays of $p_c$ that land at its critical value $c$.
\end{proposition}

We also say that $\theta$ is parabolic if there exists $n\geq1$ such that $2^n\theta=\theta$. The following is classical (see \cite{orsay1,orsay2} or \cite{Schleicher}).

\begin{proposition}\label{comptage_about_para}
Pick any parameter $c$ for which $p_c$ admits a parabolic cycle. Either $c=1/4$, in which case exactly one external ray of $\Mand$ lands at $c$, or exactly two external rays of $\Mand$ land at $c$. Furthermore, the corresponding impressions are reduced to singletons.
\end{proposition}

We shall now give a short proof of the following toy-model for Theorem~\ref{tm:contlanding} (see Section~\ref{sec:cont} for a more detailed proof of this result).

\begin{theorem}\label{tm:contlandMand}
There exists a set $\mathsf{C}\subset\R/\Z$ of full Lebesgue measure such that
\begin{enumerate}
 \item the map $\Phi^{-1}:]1,+\infty[\times\R/\Z\longrightarrow\C\setminus\Mand$ extends continuously to $\{0\}\times\mathsf{C}$,
\item the set $\mathsf{C}$ contains the Misiurewicz and parabolic angles.
 \end{enumerate}
 \end{theorem}
 
 \begin{proof}
 From \cite{Smirnov}, we know that \emph{Collet-Eckmann} angles have full Lebesgue measure and that their impression contains the limit of the corresponding ray which is a Collet-Eckmann parameter. Pick such a $\theta$, let $\Theta:=\{\theta,\theta+\frac{1}{2}\}$ and let $c_0$ be a Collet-Eckmann parameter contained in $\mathcal{I}_\Mand(\Theta)$. According to \cite[Theorem 1]{kiwireal} and \cite[\S 2]{Smirnov}, any parameter $c$ in $\mathcal{I}_\Mand(\Theta)$ has locally connected Julia set $\J_c=\K_c$ and all its cycles are repelling. By Theorem~\ref{tm:kiwi}, this implies that $p_c$ and $p_{c_0}$ are topologically conjugate on their Julia sets. By~\cite[Corollary C]{PR}, they are affine conjugate, hence $c=c_0$. Since $\mathcal{I}_\Mand(\Theta)$ is connected, it is reduced to a singleton hence $\Phi^{-1}$ extends continuously to $\{(0,\theta)\}$. 
 
 Item $2$ follows directly from the two above Propositions.
 \end{proof}

\subsection{Distribution of Misiurewicz and Parabolic parameters}
For any integers $n>k>1$, we let
\[\mathsf{C}(n,k):=\{\theta\in\R/\Z\, ; \ 2^{n-1}\theta=2^{k-1}\theta \ \text{ and } 2^{n-k}\theta\neq\theta\}\]
and we let $d(n,k):=\textup{Card}(\mathsf{C}(n,k))=2^{n-1}-2^{k-1}-2^{n-k}+1$. For any $n\geq1$, let also
\[\mathsf{P}(n):=\{\theta\in\R/\Z\, ; \ 2^{n}\theta=\theta\}~.\]
We now aim at proving the following, using Theorem~\ref{tm:contlandMand}.

\begin{theorem}\label{tm:distribparMand}
For any integer, let $\mu_n$ be the measure equidistributed on the set $X_n$ of roots of hyperbolic components of period $k|n$. Then $\mu_n$ converges to $\mu_\Mand$ in the weak sense of measures on $\C$ as $n\to\infty$.
\end{theorem}

\begin{proof}
 Let $\ell:\R/\Z\longrightarrow\Mand$ be the landing map of rays, i.e. the radial limit almost everywhere of the map $\Phi^{-1}$. It is known that it is a well-defined measurable map which satisfies $\mu_\Mand=\ell_*(\lambda_{\R/\Z})$ (see e.g.~\cite{graczykswiatek}). By the above Theorem~\ref{tm:contlandMand}, it restricts as a continuous function on a set of full measure which contains the set $\mathsf{P}(n)$ for any $n$. 

It is clear that the sequence $\{\mathsf{P}(n)\}_n$ is equidistributed. Let $\rho_n$ be the probability measure equidistributed on $\mathsf{P}(n)$. According to Theorem~\ref{tm:measure}, the above implies that
\[\ell_*(\rho_n)=\frac{1}{\textup{Card}(\mathsf{P}(n))}\sum_{\ell(\mathsf{P}(n))}\mathcal{N}_\Mand(c)\cdot \delta_c\]
converges weakly to $\ell_*(\lambda_{\R/\Z})=\mu_\Mand$, where $\mathcal{N}_\Mand(c)$ is the number of external rays of $\Mand$ that land at $c$. 
Remark now that $\textup{Card}(\mathsf{P}(n))=2^n-1$. Using Proposition~\ref{comptage_about_para}, we deduce that $\textup{Card}(X_n)=2^{n-1}$ and  
\[\ell_*(\rho_n)-\mu_n=\frac{1}{2^n-1}\mu_n-\frac{1}{2^{n-1}}\delta_{1/4}\]
converges weakly to $0$. This concludes the proof.
\end{proof}

Remark that, for any $\lambda\in\C$, it is known that the set of parameters $c\in\C$ for which $p_c$ admits a $n$-cycle of multiplier $\lambda$ equidistribute towards $\mu_\Mand$ by~\cite{multipliers}.

The same proof as above gives the following.

\begin{theorem}\label{tm:distribmisMand}
Pick any sequence $1<k(n)<n$ and let $d_n:=d(n,k(n))$. Let also $\nu_n$
\[\nu_n:=\frac{1}{d_n}\sum_{\ell(\mathsf{C}(n,k(n)))}\mathcal{N}_\Mand(c)\cdot \delta_c~,\]
where $\mathcal{N}_\Mand(c)$ is the number of external rays of $\Mand$ that land at $c$. Then $\nu_n$ converges to $\mu_\Mand$ in the weak sense of measures on $\C$ as $n\to\infty$.
\end{theorem}

\part{In the moduli space of polynomials}\label{Part:modd}

\section{Misiurewicz and Parabolic combinatorics}\label{sec:combinPd}

We define the map $M_d:\R/\Z\longrightarrow\R/\Z$ by letting
\[M_d(\theta)=d\cdot \theta \mod1.\]
We say that a combinatorics $\Theta=(\Theta_0,\ldots,\Theta_{d-2})\in\mathsf{Cb}$ is \emph{Misiurewicz} if any $\alpha\in\bigcup_i\Theta_i$ is strictly preperiodic under the map $M_d:\R/\Z\rightarrow\R/\Z$. We denote by $\mathsf{Cb}_{\textup{mis}}$ the set of all Misiurewicz combinatorics.

Similarly, we will say that a combinatorics $\Theta=(\Theta_0,\ldots,\Theta_{d-2})\in\mathsf{Cb}$ is \emph{parabolic} if for all $j$, there exists $\theta_j\in\Theta_j$ which is periodic for $M_d$. We also define the set $\mathsf{Cb}_{\textup{par}}$ as follows:
\[\mathsf{Cb}_{\textup{par}}:=\{\Theta=(\Theta_0,\ldots,\Theta_{d-2})\in\mathsf{Cb} \ \text{which is parabolic}\}.\]
Notice that if $\theta_j$ is $n_j$-periodic and $\gcd(n_i,n_j)=1$, then $\Theta\in\mathsf{Cb}_0$.

\subsection{Misiurewicz combinatorics: counting coinciding impressions}

We will use the following (see~\cite[Theorem 5.3]{kiwi-portrait}).

\begin{theorem}[Kiwi]\label{tm:kiwipcf}
The impression of a Misiurewicz combinatorics is reduced to a singleton and corresponds to the only degree $d$ critically marked Misiurewicz polynomial with the chosen combinatorics.
\end{theorem}

As noticed by Dujardin and Favre \cite[Theorem 7.18]{favredujardin}, this induces a bijection between $\mathsf{Cb}_{\textup{mis}}$ and the set of Misiurewicz parameters in the moduli space of \emph{combinatorially} marked degree $d$ polynomials (see also~\cite[Theorem III]{BFH}).

~

We now want to describe how many Misiurewicz combinatorics can have the same impression in $\mathcal{P}_d$. 
To this aim, for any $0\leq i\leq d-2$ and any $0\leq n<m$, we let
\[\mathsf{C}_i(m,n):=\{\Theta\in\mathsf{Cb}\, ; \  \Theta_i=\{\alpha_1,\ldots,\alpha_{k_i}\} ,  \ d^m\alpha_j=d^n\alpha_j\, \ \forall j\}~.\]
Relying on a result of Schleicher~\cite{Schleicher}, we can prove

\begin{proposition}\label{prop:Schleicher}
Pick any two $(d-1)$-tuples of positive integers $(n_{0},\ldots,n_{d-2})$ and $(m_{0},\ldots,m_{d-2})$ such that $m_i>n_i$. Let also $(P,c_0,\ldots,c_{d-2})\in\mathcal{P}_d$ be such that $P^{n_i}(c_i)=P^{m_i}(c_i)$, $P^{m_i-n_i}(c_i)\neq c_i$ and $P^{n_i}(c_i)$ is exactly $(m_i-n_i)$-periodic. Set
\[\mathcal{N}_{\mathsf{Cb}}(P):=\textup{Card}\left(\left\{\Theta\in\mathsf{Cb}_{\textup{mis}}\, ; \ \{(P,c_0,\ldots,c_{d-2})\}=\mathcal{I}_{\mathcal{C}_d}(\Theta)\right\}\right).\]
Then $\mathcal{N}_{\mathsf{Cb}}(P)$ is finite. More precisely, if $\Theta=(\Theta_0,\ldots,\Theta_{d-2})$ and $q_i$ is the exact period of the cycle contained in the orbit $\{M_d^k(\Theta_i)\}_{k\geq1}$, then $(m_i-n_i)|q_i$ and
\[\prod\deg_{P(c_i)}(P^{n_i-1})\cdot\left(\frac{q_i}{m_i-n_i}\right)\leq \mathcal{N}_{\mathsf{Cb}}(P)\leq\prod\deg_{P(c_i)}(P^{n_i-1})\cdot\max\left(2,\frac{q_i}{m_i-n_i}\right),\]
where the product ranges over the set of geometrically distinct critical points of $P$. 
\end{proposition}

The proof of \cite[Lemma 2.4]{Schleicher} directly gives the next lemma. Notice that the periods and preperiods don't depend on the critical portraits, i.e. for any $\Theta,\Theta'$ which impression coincide, the periods and preperiods of $\Theta_i$ and $\Theta_i'$ coincide.

\begin{lemma}\label{lm:Sch}
Let $P$ be any degre $d\geq2$ polynomial. Let $z$ be a repelling or parabolic periodic point of $P$ of exact period $k\geq1$. At least one dynamical ray lands at $z$ and:
\begin{enumerate}
\item If at least three periodic rays land at $z$, then the first return map $P^k$ permutes transitively those dynamical rays,
\item If exactly two periodic rays land at $z$, then either the first return map $P^k$ permutes transitively those dynamical rays, or it fixes each of them.
\end{enumerate}
Moreover, the number of landing rays is constant along the forward orbit of $z$.
\end{lemma}

\begin{proof}[Proof of Proposition~\ref{prop:Schleicher}]
By Theorem~\ref{tm:kiwipcf}, if $\alpha$ lies in the orbit under iteration of $M_d$of $\Theta_i$, then the point $z$ at which it lands lies in the orbit under iteration of $P$ of $c_i$. In particular, if $R_\alpha$ is the dynamical ray of angle $\alpha$ of $P$, then $P^{q_i}(R_\alpha)=R_\alpha$, hence $P^{q_i}(z)=z$, i.e. $(m_i-n_i)|q_i$.

Up to reordering, write now $c_0,\ldots,c_k$ the number of geometrically distinct critical points of $P$, $d_0,\ldots,d_k\geq2$ the local degree of $P$ at $c_0,\ldots,c_k$ respectively. As long as $P^k(c_i)$ is not a critical point, a ray landing at $P^{k+1}(c_i)$ has one and only one preimage under $P$ which lands at $P^k(c_i)$. On the other hand, if $P^k(c_i)=c_j$ for some $j\neq i$, then any ray landing at $P^{k+1}(c_i)$ has exactly $d_j$ preimages landing at $P^k(c_i)=c_j$. As a conclusion, the number $N_i$ of rays landing at $P(c_i)$ is exaclty $\deg_{P(c_i)}(P^{n_i-1})$ times the number of rays landing at $P^{n_i}(c_i)$, which satisfies
\[\deg_{P(c_i)}(P^{n_i-1})\cdot\left(\frac{q_i}{m_i-n_i}\right)\leq N_i \leq\deg_{P(c_i)}(P^{n_i-1})\cdot\max\left(2,\frac{q_i}{m_i-n_i}\right).\]
Finally, each ray landing at $P(c_i)$ has exactly $d_i$ preimages. For any $i$, pick $\theta_i$ landing at $P(c_i)$ and let $\Theta_i$ be the set of angles whose ray lands at $c_i$ and $M_d(\alpha)=\theta_i$ for any $\alpha\in\Theta_i$. Then $\Theta:=(\Theta_0,\ldots,\Theta_{d-2})$ (with repetitions if critical points are multiple) is a critical portrait for $P$ and we can associate to each collection $(\theta_0,\ldots,\theta_{d-2})$ of angles landing respectively at $P(c_i)$ one and only one critical portrait for $P$. The conclusion then follows from Lemma~\ref{lm:Sch}.
\end{proof}

\subsection{Parabolic combinatorics: a landing property}
We need the following definition.

\begin{definition}
Let $P$ be a degree $d$ polynomial. We say that a parabolic periodic point $z$ of $P$ is $n$-\emph{degenerate} if it has period $k|n$ and if $n$ is minimal so that $(P^n)'(z)=1$.
\end{definition}

The aim of the present section is to prove the following result.

\begin{theorem}\label{tm:rational}
Pick $\un=(n_0,\ldots,n_{d-2})$ with $\gcd(n_i,n_j)=1$ for $i\neq j$. Let $\Theta=(\Theta_0,\ldots,\Theta_{d-2})\in\mathsf{Cb}_{\textup{par}}$ be a portrait such that for any $j$, there exists $\theta_j\in\Theta_j$ which is exactly $n_j$-periodic for $M_d$. Then $\mathcal{I}_{\mathcal{C}_d}(\Theta)$ consists in a single critically marked polynomial $(P,c_0,\ldots,c_{d-2})$ having $d-1$ distinct parabolic periodic cycles which are respectively $n_j$-\emph{degenerate}. Moreover, $\theta_j$ lands at a parabolic point of period $k|n_j$.
\end{theorem}

For our proof, we deeply rely on the seminal work \cite{orsay1} of Douady and Hubbard. Moreover, we follow closely the proof of \cite[Expos\'e VIII Th\'eor\`eme 2]{orsay1}.  Let us first make some preliminaries.

Recall the following (see~\cite[p. 225]{Silverman}, \cite[Appendix D]{milnor3} or \cite[Theorem 2.1]{BB2}):
\begin{theorem}[Milnor, Silverman]\label{tmpern}
For any $n\geq1$, there exists a polynomial map $p_n:\mathcal{P}_d\times \C\to\C$ such that for any $(P,c_0,\ldots,c_{d-2})\in\mathcal{P}_d$ and any $w\in\C$,
\begin{enumerate}
\item  if $w\neq1$, then $p_n(P,w)=0$ if and only if $P$ has a cycle of exact period $n$ and multiplier $w$,
\item otherwise, $p_n(P,1)=0$ if and only if there exists $q\geq1$ such that $P$ has a cycle of exact period $n/q$ and multiplier $\eta$ a primitive $q$-root of unity.
\end{enumerate}
\end{theorem}
We now define an algebraic hypersurface by letting
$$\Per_n(w):=\{(P,c_0,\ldots,c_{d-2})\in \mathcal{P}_d \ | \ p_n(P,w)=0\}~,$$
for $n\geq1$ and $w\in\C$. By the Fatou-Shishikura inequality and using the compactness of the connectedness locus, we have the following:

\begin{lemma}\label{lm:finite}
Pick $n_0,\ldots, n_{d-2}\geq1$ and assume that $\gcd(n_i,n_j)=1$ for all $i\neq j$. Then, for any $w_0,\ldots,w_{d-2}\in\overline{\D}$, the algebraic variety $\bigcap_i\Per_{n_i}(w_i)$ is a finite set.
\end{lemma}

We are now in position to prove Theorem~\ref{tm:rational}.

\begin{proof}[Proof of Theorem~\ref{tm:rational}]
Pick $(P,c_0,\ldots,c_{d-2})\in\mathcal{I}_{\mathcal{C}_d}(\Theta)$. Recall that $\mathcal{I}_{\mathcal{C}_d}(\Theta)\subset\mathcal{C}_d$ so that the B\"ottcher coordinate of $P$ at infinity is a biholomorphism $\phi_P:\C\setminus\K_P\longrightarrow\C\setminus\overline{\D}$. According to \cite[Expos\'e VIII, \S 2, Proposition 2]{orsay1}, the dynamical external rays of $P$ of respective angles $\theta_1,\ldots,\theta_{d-1}$ land in the dynamical plane to periodic points $z_0,\ldots,z_{d-2}$ of $P$. Moreover, the period of $z_i$ divides $n_i$ and either $z_i$ is repelling, or $(P^{n_i})'(z_i)=1$.

First, notice that, since $\gcd(n_i,n_j)=1$ for $i\neq j$, the points $z_i$ and $z_j$ can not lie in the same cycle. We now assume by contradiction that there exists $0\leq i\leq d-2$ such that $z_i$ is repelling. Since $\theta_i$ lands to a repelling cycle, $P^k(z_i)$ is not a critical point of $P$. Moreover, by the implicit function theorem we can follow $z_i$ holomorphically as a repelling $n_i$-periodic point $z_i(Q)$ of $Q$, in a neighborhood of $P$ in $\mathcal{P}_d$. We thus may apply \cite[Expos\'e VIII, \S 2, Proposition 3]{orsay1}: there exist a neighborhood $W$ of $(P,c_0,\ldots,c_{d-2})$ in $\mathcal{P}_d$ and a continuous map
\[\psi:(Q,s)\in W\times\R_+\longmapsto \psi(Q,s)\in\C\]
which depends holomorphically of $Q\in W$ and such that the following holds
\begin{itemize}
\item for any $s\geq0$ and any $Q\in W$, $\psi(Q,s)=\phi_Q^{-1}\left(e^{s+2i\pi\theta_i}\right)$ and in particular $g_Q\left(\psi(Q,s)\right)=s$,
\item for any $Q\in W$, the dynamical ray of $Q$ of angle $\theta_i$ lands at $z_i(Q)=\psi(Q,0)$.
\end{itemize}
According to \cite[Lemma 3.19]{kiwi-portrait}, the visible shift locus is dense in the shift locus, and since $(P,c_0,\ldots,c_{d-2})\in\mathcal{I}_{\mathcal{C}_d}(\Theta)$, $W\cap\mathcal{S}_d\neq\emptyset$ and $P\in \overline{W\cap\mathcal{S}_d^{\textup{vis}}}$.

Pick now $Q_n\in\mathcal{S}_d^{\textup{vis}}\cap W$ and $s_n>0$ such that $Q_n\rightarrow P$ and $s_n\rightarrow0$ as $n\rightarrow+\infty$. We then have $\Theta(Q_n)\rightarrow\Theta$ and $g_{Q_n}(c_i(Q_n))=s_n\rightarrow0$. More precisely, we have $\Theta_i(Q_n)\longrightarrow \Theta_i$ and
\[\psi(Q_n,s_n)-c_i(Q_n)\longrightarrow_{n\rightarrow\infty}0,\]
i.e. $c_i(P)=z_i$, which is a contradiction since $z_i\in\partial\K_P$.

We have shown that $\mathcal{I}_{\mathcal{C}_d}(\Theta)$ is contained in the algebraic variety $\bigcap_i\Per_{n_i}(1)$ which, owing to Lemma~\ref{lm:finite}, is finite. Since $\mathcal{I}_{\mathcal{C}_d}(\Theta)$ is a connected compact set included in a finite set, it is reduced to a single point.
\end{proof}

\section{The bifurcation measure and combinatorics}\label{sec:bifmes}

\subsection{The bifurcation measure and the Goldberg and landing maps}

We recall here material from \cite[\S 6 $\&$ 7]{favredujardin}. Recall that we defined the psh and continuous function $G:\mathcal{P}_d\longrightarrow\R_+$ by letting $G(P):=\max_{0\leq j\leq d-2} g_P(c_j)$ for any $(P,c_0,\ldots,c_{d-2})\in\mathcal{P}_d$. We can define the \emph{bifurcation measure} $\mu_\bif$ of the moduli space $\mathcal{P}_d$ as the Monge-Amp\`ere mass of the function $G$, i.e.

\[\mu_\bif:= (dd^cG)^{d-1}~.\]

This measure was introduced first by Dujardin and Favre~\cite{favredujardin} and they proved that it is a probability measure which is supported by the Shilov boundary of the connectedness locus $\partial_S\mathcal{C}_d$ (see \cite[\S 6]{favredujardin}).

\paragraph*{The Goldberg and landing maps after Dujardin and Favre} 
For $r>0$, let $\mathcal{G}(r):=\{(P,c_0,\ldots,c_{d-2})\in \mathcal{P}_d \, ; \ g_P(c_i)=r, \, \forall 0\leq i\leq d-2\}$. The set $\mathcal{G}(r)$ is contained in $\mathcal{S}_d^{\textup{vis}}$. Moreover, there exists a unique continuous map
\begin{eqnarray*}
\Phi_g:\mathsf{Cb} \times \R^+_* & \longrightarrow & \mathcal{P}_d\\
(\Theta,r) & \longmapsto & \left(P(\Theta,r),c_0(\Theta,r),\ldots,c_{d-2}(\Theta,r)\right)
\end{eqnarray*}
such that the following holds:
\begin{itemize}
\item $P(\Theta,r)\in\mathcal{S}_d^{\textup{vis}}$ and the $(d-1)$-tuple $\Theta$ of subsets is the combinatorics of $P(\Theta,r)$ and $g_{P(\Theta,r)}(c_i(\Theta,r))=r$ for each $0\leq i\leq d-2$,
\item the map $\Phi_g(\cdot,r)$ is a homeomorphism from $\mathsf{Cb}$ onto $\mathcal{G}(r)$. Moreover, $\Phi_g(\cdot,r)$ restricts to a homeomorphism from $\mathsf{Cb}_0$ onto the subset of $\mathcal{G}(r)$ of polynomials for which all critical points are simple.
\end{itemize}
The map $\Phi_g$ is the \emph{Goldberg} map of the moduli space $\mathcal{P}_d$. The radial limit of the map $\Phi_g(\cdot,r)$ as $r\rightarrow0$ exists $\mu_{\mathsf{Cb}}$-almost everywhere and defines a map $e:\mathsf{Cb}\longrightarrow\mathcal{P}_d$. By construction, its image is contained in $\partial\mathcal{C}_d\cap \partial\mathcal{S}_d^{\textup{vis}}$.

\begin{definition}
The map $e:\mathsf{Cb}\longrightarrow\mathcal{P}_d$ is called the \emph{landing map}.
\end{definition}

The main result relating this landing map with the bifurcation measure is the following (see~\cite[Theorem 9]{favredujardin}).

\begin{theorem}[Dujardin-Favre]\label{tm:landing}
$e_*\left(\mu_{\mathsf{Cb}}\right)=\mu_\bif$.
\end{theorem}

\subsection{Continuity of the landing map on a set of $\mu_{\mathsf{Cb}}$-full measure}\label{sec:cont}

 The main goal of this section is to prove the following result.

\begin{theorem}\label{tm:contlanding}
There exists a set $\mathsf{Cb}_1\subset\mathsf{Cb}_0$ of full $\mu_{\mathsf{Cb}}$-measure such that the map $e|_{\mathsf{Cb}_1}$ is continuous. Moreover, the set $\mathsf{Cb}_1$ contains the totally parabolic combinatorics $\mathsf{Cb}_{\textup{par}}$ and Misiurewicz combinatorics $\mathsf{Cb}_{\textup{mis}}$.
\end{theorem}

In fact, we rely on the stronger statement below, which is essentially the combination of Theorem~\ref{tm:kiwi} with \cite[Theorem 1]{kiwi-portrait} and with the rigidity property established in \cite[Corollary C]{PR}.

\begin{theorem}\label{tm:combcont}
Pick $\Theta\in\mathsf{Cb}_0$ such that there exists $(P,c_0,\ldots,c_{d-2})\in \mathcal{I}_{\mathcal{C}_d}(\Theta)$ with $\J_P=\K_P$ and which satisfies the TCE condition. Then the impression $\mathcal{I}_{\mathcal{C}_d}(\Theta)$ is reduced to a singleton.
\end{theorem}

\begin{proof}
Pick $\Theta\in\mathsf{Cb}$ and $(P,c_0,\ldots,c_{d-2}),(\tilde P,\tilde c_0,\ldots,\tilde c_{d-2})\in\mathcal{I}_{\mathcal{C}_d}(\Theta)$ such that $(P,c_0,\ldots,c_{d-2})$ satisfies the TCE condition. According to \cite[Theorem 1]{kiwi-portrait}, the real lamination of $P$ is equal to that of $\Theta$ and has aperiodic kneading since $P$ has only repelling cycles. Again by \cite[Theorem 1]{kiwi-portrait}, $P$ and $\tilde P$ have the same real lamination and do not satisfy the \emph{Strongly Recurrent Condition} (see e.g. \cite[\S 2]{Smirnov}). In particular, $\tilde{P}$ also has only repelling periodic points and its Julia set is locally connected. Moreover, $\tilde P\in\mathcal{C}_d\cap\partial\mathcal{S}_d$ and all its cycle are repelling, hence $\K_{\tilde P}$ has no interior, i.e. $\J_{\tilde P}=\K_{\tilde P}$.

We now apply Theorem~\ref{tm:kiwi}: the polynomials $P$ and $\tilde P$ are conjugate on their Julia sets by an orientation preserving homeomorphism. Finally, since $P$ satisfies the TCE property and $C(P)\subset\J_P=\K_P$, \cite[Corollary C]{PR} states that $P$ and $\tilde P$ are affine conjugate and there exists $\sigma\in\mathfrak{S}_{d-1}$ such that $\tilde c_i=c_{\sigma(i)}$. Hence $\mathcal{I}_{\mathcal{C}_d}(\Theta)$ is contained in a finite subset of $\mathcal{P}_d$. 

Since $\mathcal{I}_{\mathcal{C}_d}(\Theta)$ is connected, it is reduced to a singleton.
\end{proof}

We now are in position to prove Theorem~\ref{tm:contlanding}.

\begin{proof}[Proof of Theorem~\ref{tm:contlanding}]
Dujardin and Favre~\cite[Theorem 10]{favredujardin} prove that there exists a Borel set $\mathsf{Cb}_1^*\subset\mathsf{Cb}_0$ such that
\begin{itemize}
\item $\mathsf{Cb}_1^*$ has full $\mu_{\mathsf{Cb}}$-measure,
\item for any $\Theta\in \mathsf{Cb}_1^*$ the impression $\mathcal{I}_{\mathcal{C}_d}(\Theta)$ contains a polynomial $P$ satisfying the TCE condition.
\end{itemize}
Let us now set
\[\mathsf{Cb}_1:=\mathsf{Cb}_1^*\cup\mathsf{Cb}_{\textup{mis}}\cup\mathsf{Cb}_{\textup{par}}.\]
 Pick $\Theta\in\mathsf{Cb}_1$. According to Theorem~\ref{tm:combcont}, Theorem~\ref{tm:kiwipcf} of Kiwi and Theorem~\ref{tm:rational}, the impression $\mathcal{I}_{\mathcal{C}_d}(\Theta)$ is reduced to a singleton. By definition of the impression $\mathcal{I}_{\mathcal{C}_d}(\Theta)$, the map $\Phi_g$ extends continuously to $\mathsf{Cb}_1\times\{0\}$.
Recall that the landing map $e$ is the radial limit almost everywhere of the map $\Phi_g(\cdot,r)$, as $r\to 0$. The landing map $e$ thus coincides $\mu_{\mathsf{Cb}}$-almost everywhere with the extension of the Goldberg map $\Phi_g$, which ends the proof.
\end{proof}

\section{Distribution of Misiurewicz and Parabolic Combinatorics}

Our goal here is to apply the combinatorial tools studied above to equidistribution problems concerning Misiurewicz parameters with prescribed combinatorics.

\subsection{Preliminary properties}

Recall that, for any $0\leq i\leq d-2$ and any $0\leq n<m$, we have denoted 
\[\mathsf{C}_i(m,n):=\{\Theta\in\mathsf{Cb}\, ; \  \Theta_i=\{\alpha_1,\ldots,\alpha_{k_i}\} ,  \ d^m\alpha_j=d^n\alpha_j\, \ \forall j\}~.\]
For any $i$, pick any sequences $0<n_{k,i}<m_{k,i}$ such that $m_{k,i}\rightarrow\infty$ as $k\rightarrow\infty$ and let
\[\mathsf{C}_{k,i}^*:=\mathsf{C}_i(m_{k,i},n_{k,i})\setminus\mathsf{C}_i(m_{k,i}-n_{k,i},0) \ \text{and} \ \mathsf{C}_k^*:=\bigcap_{i=0}^{d-2}\mathsf{C}_{k,i}^*.\]
Notice that the set $\mathsf{C}_k^*$ is finite and that $\textup{Card}(\mathsf{C}_k^*)\geq c\cdot d^{\sum_im_{k,i}}$, where $c>0$ is a constant depending only on $d$ and not on the sequences $(m_{k,i})$ and $(n_{k,i})$ (see \cite[\S5.3]{favregauthier}).  Finally, we let $\nu_k$ be the probability measure on $\mathsf{Cb}$ which is equidistributed on $\mathsf{C}_k^*$.

\begin{lemma}\label{lm:nocollapsing}
The sequence $\nu_k(\mathsf{Cb}\setminus\mathsf{Cb}_0)$ converges to $0$ as $k\rightarrow+\infty$.
\end{lemma}

\begin{proof}
To do so, it is sufficient to prove that 
\[\limsup_{k\rightarrow+\infty}\frac{\textup{Card}\left(\mathsf{C}_k^*\setminus\mathsf{Cb}_0\right)}{\textup{Card}\left(\mathsf{C}_k^*\right)}=0~.\]
The set $\mathsf{C}_k^*\setminus\mathsf{Cb}_0$ coincides with the union over the $j_k$ of the set $\bigcap_{i\neq j_k}\mathsf{C}_{k,i}^*$ intersected with $\bigcup_{i\neq j_k}\{\Theta\in\mathsf{Cb}\, ; \ \Theta_{j_k}=\Theta_i\}$. As a consequence, 
\[\textup{Card}\left(\mathsf{C}_k^*\setminus\mathsf{Cb}_0\right)\leq C \sum_{j_k} d^{\sum m_{k,i}-m_{k,j_k}}\]
where $C$ depends only on $d$. Hence
\[\frac{\textup{Card}\left(\mathsf{C}_k^*\setminus\mathsf{Cb}_0\right)}{\textup{Card}\left(\mathsf{C}_k^*\right)}\leq \sum_{j_k} \frac{C}{c}d^{-m_{k,j_k}}\longrightarrow0~,\]
as $k\rightarrow+\infty$, which ends the proof.
\end{proof}

We now give a more precise description of the spaces $\mathsf{Cb}$ and $\mathsf{S}$ and of the measure $\mu_{\mathsf{Cb}}$ we will need in our proof. We refer to~\cite[\S7.1]{favredujardin} for more details.

The set $\mathsf{Cb}$ has $\mathsf{Cb}_0$ as an open and dense subset. Moreover, $\mathsf{Cb}_0$ can also be seen as an open subset of the set $\mathsf{S}^{d-1}$. The set $\mathsf{S}$ is a translation manifold of dimension $1$ which has $\lfloor d/2\rfloor$ connected components, each of them being isomorphic to $\R/\Z$. 

We endow each of these components with a copy of the probability measure $\lambda_{\R/\Z}$ and let $\lambda_{\mathsf{S}}$ be the probability measure which is proportional to the obtained finite measure.
Notice that $\lambda_{\mathsf{S}}^{\otimes(d-1)}(\mathsf{Cb}_0)>0$ and let $\mu_{\mathsf{Cb}_0}$ be the measure 
\[\mu_{\mathsf{Cb}_0}:=\frac{1}{\lambda_{\mathsf{S}}^{\otimes(d-1)}(\mathsf{Cb}_0)}\mathbf{1}_{\mathsf{Cb}_0}\cdot\lambda_{\mathsf{S}}^{\otimes(d-1)}.\]
The measure $\mu_{\mathsf{Cb}}$ is then the trivial extension of $\mu_{\mathsf{Cb}_0}$ to $\mathsf{Cb}$.

\subsection{Equidistribution results: Theorems~\ref{tm:distrib} and~\ref{tm:distribpara}}

For any $0\leq i\leq d-2$ and any $0\leq n<m$, we let 
\[\mathsf{S}(m,n):=\{\{\alpha,\alpha'\}\in\mathsf{S}\, ; \  d^m\alpha=d^n\alpha\} \ \textup{and} \ \mathsf{S}^*(m,n):=\mathsf{S}(m,n)\setminus\mathsf{S}(m-n,0) ~.\]
For any $i$, pick any sequences $0<n_{k,i}<m_{k,i}$ such that $m_{k,i}\rightarrow\infty$ as $k\rightarrow\infty$ and let
\[\mathsf{S}_{k}^*:=\prod_{i=0}^{d-2}\mathsf{S}^*(m_{k,i},n_{k,i})\subset\mathsf{S}^{d-1}.\]
Finally, we let $m_k$ be the probability measure on $\mathsf{S}^{d-1}$ which is supported on $\mathsf{S}_k^*$ and we let $\lambda_{\mathsf{S}}$ be the natural probability measure on $\mathsf{S}$.

As in~\cite[\S7.1]{favredujardin}, for any collection of open intervals $I_0,\ldots,I_{d-2}\subset\R/\Z$, and any collection of integers $q_0,\ldots,q_{d-2}\in\{1,\ldots,\lfloor d/2\rfloor\}$, wet let 
\[\mathsf{I}_i(q):=\left\{\{\alpha,\alpha'\}\in\mathsf{S}\, ; \ \{\alpha,\alpha'\}\subset I_i\cup\left(I_i+\frac{q_i}{d}\right)\right\}\]
and we can define an open set of $\mathsf{S}^{d-1}$ by setting
\[U(I,q):= \left\{\Theta\in\mathsf{S}^{d-1}\, ; \ \Theta_i\in \mathsf{I}_i(q_i)\right\},\]
where $I:=(I_0,\ldots,I_{d-2})$ and $q=(q_0,\ldots,q_{d-2})$. Notice that such open sets define with small intervals span the topology of $\mathsf{S}^{d-1}$.

We rely on the following key intermediate result.

\begin{lemma}\label{lm:reduction}
The sequence $(m_k)$ is equidistributed with respect to $\lambda_{\mathsf{S}}^{d-1}$ on $\mathsf{S}^{d-1}$. More precisely, if $I=(I_0,\ldots,I_{d-2})$ is a $(d-1)$-tuple of intervals and any $(d-1)$-tuple of integers $q=(q_0,\ldots,q_{d-2})$ with $1\leq q_i\leq \lfloor d/2\rfloor$, we have
\[\lim_{k\rightarrow\infty}m_k(U(I,q))=\lambda_{\mathsf{S}}^{\otimes(d-1)}(U(I,q))=\prod_{i=0}^{d-2}\lambda_{\mathsf{S}}(I_i(q_i)).\]
\end{lemma}

\begin{proof}
As the measure $m_k$ is a product measure $m_k=m_{k,0}\otimes\cdots\otimes m_{k,d-1}$, where $m_{k,i}$ is the probability measure equidistributed on the set $S^*(m_{k,i},n_{k,i})$, by Fubini Theorem, is is sufficient to prove that $m_{k,i}$ is equidistributed with respect to $\lambda_{\mathsf{S}}$ as $k\rightarrow+\infty$.

Let $d_{k}:=\textup{Card}(\mathsf{S}^*(m_{k,i},n_{k,i}))$. Since for any $m>n>0$,
\[d^{m}-d^{n}\leq \textup{Card}(\mathsf{S}(m,n))\leq  \lfloor d/2\rfloor \cdot (d^{m}-d^{n})~,\]
we find 
\begin{eqnarray*}
d_k & = & \textup{Card}(\mathsf{S}(m_{k,i},n_{k,i}))-\textup{Card}(\mathsf{S}(m_{k,i}-n_{k,i},0))\\
& \geq & d^{m_{k,i}}-d^{n_{k,i}}- \lfloor d/2\rfloor \cdot (d^{m_{k,i}-n_{k,i}}-1)\\
& \geq & d^{m_{k,i}}-d^{n_{k,i}}- \frac{d/2+1}{d} \cdot (d^{m_{k,i}}-d^{n_{k,i}})\\
& \geq & \left(1- \frac{d/2+1}{d}\right)\cdot (d^{m_{k,i}}-d^{n_{k,i}}).
\end{eqnarray*}
Notice that $1- \frac{d/2+1}{d}>0$. Now, the natural measure $\lambda_{\mathsf{S}}$ is the renormalization of $\lfloor d/2\rfloor$ copies of $\lambda_{\R/\Z}$, hence we can directly apply Lemma~\ref{lm:equidsitrcircle}. This gives the equidistribution of $m_{k,i}$ with respect to $\lambda_{\mathsf{S}}$, as $k\rightarrow+\infty$ and the proof is complete.
\end{proof}

As a consequence, using classical measure theory, we easily get the following:

\begin{corollary}\label{cor:nocollapsing}
The sequence $(m_k)$ converges towards $\lambda^{\otimes(d-1)}_{\mathsf{S}}$ in the weak sense of probability measures on $\mathsf{S}^{d-1}$.
\end{corollary}

We now can end the proof of Theorem~\ref{tm:distrib}.

\begin{proof}[Proof of Theorem~\ref{tm:distrib}]
Write again $\lambda:=\lambda^{\otimes(d-1)}_{\mathsf{S}}$. Recall that $\nu_k$ is the probability measure equidistributed on $\mathsf{C}_k^*$ and $\mu_k$ is the measure defined in Theorem~\ref{tm:distrib}. By Theorem~\ref{tm:kiwipcf}, one has $e_*(\nu_k)=\mu_k$ for any $k$. Notice also that $\mu_\bif=e_*(\mu_{\mathsf{Cb}})$, by Theorem~\ref{tm:landing}. According to  Theorem~\ref{tm:contlanding} and Theorem~\ref{tm:measure}, it is sufficient to prove that $(\nu_k)$ converges weakly to $\mu_{\mathsf{Cb}}$.

First, remark that, since $\mu_{\mathsf{Cb}}$ is the trivial extension of $\mu_{\mathsf{Cb}_0}$ to $\mathsf{Cb}$, Lemma~\ref{lm:nocollapsing} implies that it is actually sufficient to prove that $\nu_k$ converges weakly towards $\mu_{\mathsf{Cb}_0}$ to conclude.
Let $K$ be any compact subset of $\mathsf{Cb}_0$. Then
\[\nu_k(K)-\mu_{\mathsf{Cb}_0}(K)=\frac{m_k(K)}{m_k(\mathsf{Cb}_0)}-\frac{\lambda(K)}{\lambda(\mathsf{Cb}_0)}.\]
According to Corollary~\ref{cor:nocollapsing} and to the Fact of Section~\ref{sec:measuretheory}, for any $\epsilon>0$, there exists $k_0\geq1$ such that for any $k\geq k_0$,
\[m_k(K)\leq \lambda(K)+\epsilon \ \text{and} \ m_k(\mathsf{Cb}_0)\geq \lambda(\mathsf{Cb}_0)-\epsilon\]
since $\mathsf{Cb}_0$ is open and $K$ is compact in $\mathsf{Cb}_0$, hence in $\mathsf{Cb}$. In particular, for $k\geq k_0$, we find
\[\nu_k(K)-\mu_{\mathsf{Cb}_0}(K)\leq \epsilon\cdot\frac{\lambda(\mathsf{Cb}_0)+\lambda(K)}{\lambda(\mathsf{Cb}_0)(\lambda(\mathsf{Cb}_0)-\epsilon)}.\]
Taking the limsup as $k\rightarrow\infty$ and then making $\epsilon\rightarrow0$ gives 
\[\limsup_{k\rightarrow\infty}\nu_k(K)\leq\mu_{\mathsf{Cb}_0}(K).\]
This ends the proof, using again the Fact of Section~\ref{sec:measuretheory}.
\end{proof}

The proof of Theorem~\ref{tm:distribpara} is similar, so we omit it. As observed in the introduction, we lack a precise control on the cardinality of combinatorics that land at a given parabolic polynomial to have a better result in the spirit of Theorem~\ref{tm:distribparMand}.

 An easy estimate follows from Theorem~\ref{tm:rational} up to considering all the possible permutations of the given combinatorics (a priori, two permuted combinatorics may land at the same parameter). Given such a polynomial $P$ with periodic parabolic cycles of exact periods $k_i$ and combinatorial periods $n_i$, the number $\mathcal{N}_{\mathsf{Cb}}(P)$ is bounded above:
\[\mathcal{N}_{\mathsf{Cb}}(P)\leq \left((d-1)!\right)^2\cdot\prod_{i=0}^{d-2} \left(k_i\cdot \max\left\{2,\frac{n_i}{k_i}\right\}\right).\]
Following \cite{Schleicher}, one can expect that the only angles which actually belong to $\Theta_i$ are the \emph{characteristic rays} of the parabolic cycle whose parabolic basin contains $c_i$, i.e. the rays separating the petals containing $P(c_i)$ from the other petals clustering at the same point of the considered parabolic cycle. This would give an exact formula for $\mathcal{N}_{\mathsf{Cb}}(P)$.

\part{In the quadratic anti-holomorphic family}\label{Part:anti}

\section{The anti-holomorphic quadratic family}
\subsection{Anti-holomorphic polynomials and the Tricorn}
We now aim at studying the family of \emph{quadratic anti-holomorphic} dynamical systems, i.e. the family
\[f_c(z):=\bar{z}^2+c~, \ z\in\C~,\]
parametrized by $c\in\C$. It is classical to proceed by analogy with the holomorphic case, i.e. to define the filled Julia set of $f_c$ and the Julia set of $f_c$ by letting 
\[\K_c:=\{z\in\C \, ; \ (f_c^n(z))_n \ \text{is bounded}\} \ \text{ and } \ \J_c:=\partial\K_c~.\]
We also define the \emph{Tricorn} as the set 
\[\Mand_2^*:=\{c\in\C \, ; \ (f_c^n(0))_n \ \text{is bounded}\}~.\]
Again, as in the holomorphic case, for $n>k>0$, we let
\[\Per(n,k):=\{c\in\C \, ; \ f_c^n(0)=f_c^k(0)\} \ \ \text{and} \ \ \Per^*(n,k):=\{c\in \Per(n,k) \, ; \ f_c^{n-k}(0)\neq0\}.\]

\begin{definition}
We say that a parameter $c\in\bigcup_{n>k\geq1}\Per^*(n,k)$ is a \emph{Misiurewicz} parameter.
\end{definition}
Notice that we chose this definition by analogy to the holomorphic case. We now want to address the following question.

\begin{question}
Is the set $\Per(n,k(n))$ (resp. the set $\Per^*(n,k(n))$) finite and can we describe its distribution as $n\to\infty$, for any sequence $0\leq k(n)<n$?
\end{question}

The rest of the paper gives a partial answer to the above question.

~

For convenience, define the family of quadratic anti-holomorphic polynomials
\[f_\lambda(z):=\bar{z}^2+(a+ib)^2~, \ z\in\C~,\]
for $\lambda=(a,b)\in\C^2$. A classical observation is that $f_\lambda\circ f_\lambda$ defines a family of holomorphic degree $4$ polynomials $P_\lambda$. An easy computation shows that for $\lambda=(a,b)\in\C^2$, 
\[P_\lambda(z)=z^4+2(a-ib)^2z^2+(a+ib)^2+(a-ib)^4~, \ z\in\C.\]
This family has (complex) dimension $2$. The critical points of $P_\lambda$ are exactly $c_0:=0$, $c_1:=ia+b$ and $c_2=-(ia+b)$. It is also easy to check that for any $\lambda=(a,b)\in\C^2$, we have $c_1=c_2$ if and only if $c_0=c_1$ if and only if $c_0=c_2$ if and only if $b=-ia$.

\begin{lemma}\label{lm:complexTricorn}
The family $(P_\lambda)_{\lambda\in\C^2}$ projects in the moduli space $\mathcal{P}_4$ to the surface $\mathcal{X}:=\{(P,c_0,c_1,c_2)\in\mathcal{P}_4\, ; \ c_1=-c_2\}$. Moreover, the projection $\pi:\C^2\longrightarrow\mathcal{X}$ is a degree $6$ branched covering ramifying exactly at $\lambda=(0,0)$. Moreover, if $\lambda=(a,b)\in \R^2$ then $P_\lambda$ is the only real representative of $\{P_\lambda\}$ in the family.
\end{lemma}

\begin{proof}
The surface $\mathcal{X}$ is irreducible and $\pi$ is proper. So it is surjective, hence has finite degree. Solving the equation $(\alpha z+ \beta) \circ P_\lambda = P_\lambda \circ (\alpha z+ \beta)$ implies $\beta=0$ and $\alpha^3=1$. Discussing the different cases gives six solutions and ones easily sees that if $(a,b)$ in $\R^2$ then $P_\lambda$ is the only real representative of $\{P_\lambda\}$ in the family. Moreover, it is clear that $\pi$ ramifies exactly at $\lambda=(0,0)$ (else $c_1\neq 0$).
\end{proof}

 We also will rely on the following which is essentially obvious.

\begin{lemma}\label{lm:complexTricorn2}
The map $\pi|_{\R^2}:\R^2\longrightarrow\mathcal{P}_4$ is a real-analytic homeomorphism onto its image. Moreover, if $\lambda\in \R^2$, then $f_\lambda(c_1)=f_\lambda(c_2)=c_0$.
\end{lemma}

\begin{proof}
If $\lambda\in\R^2\setminus\{0\}$, the polynomials $P_\lambda$ and $P_{-\lambda}$ are affine conjugate, but the conjugacy exchanges $c_1$ with $c_2$ and the conclusion follows from Lemma~\ref{lm:complexTricorn}. The fact that $f_\lambda(c_1)=f_\lambda(c_2)=c_0$ follows from a direct computation.
\end{proof}

\subsection{The combinatorial space}

For the material of this section, we follow~\cite{Nakane}. Let $\psi_c$ be the B\"ottcher coordinate of the anti-holomorphic polynomial $f_c(z)=\bar{z}^2+c$, with $c\in\C$, i.e. the holomorphic map conjugating $f_c$ near $\infty$ to $\bar{z}^2$ near $\infty$ which is tangent to the identity. Let $\lambda=(a,b)\in\R^2$ be such that $c=(a+ib)^2$. It is known to be a biholomorphic map from $\{z\in\C \, ; \ g_{P_\lambda}(z)>g_{P_\lambda}(c_0)\}$ onto $\C\setminus\overline{D}(0,\exp(g_{P_\lambda}(c_0)))$. Moreover, we also have $\psi_c\circ P_\lambda=(\psi_c)^4$, i.e. $\psi_c=\phi_\lambda$ (recall that $\phi_\lambda$ is the B\"ottcher coordinate of $P_\lambda$). 

For $c\in\C$, we let $\Psi^*(c):=\psi_c(c)$, when $c\in \C\setminus\Mand_2^*$. The map 
\[\Psi^*:\C\setminus\Mand_2^*\longrightarrow\C\setminus\overline{\D}\]
is known to be a real-analytic isomorphism (see \cite{Nakane}).

For $\theta\in\R/\Z$, the \emph{external ray} of $\Mand_2^*$ of angle $\theta$ is the curve $\mathcal{R}^*(\theta)$ defined by
\[\mathcal{R}^*(\theta):=\left(\Psi^*\right)^{-1}\left(\{Re^{2i\pi\theta}\, ; \ 1<R<+\infty\}\right).\]

We will need the following.

\begin{lemma}\label{lm:impressionanti}
Let $c=(a+ib)^2\in\C$ with $\lambda:=(a,b)\in\R^2$. Assume that $c\in\mathcal{R}^*(\theta)$ and $\Psi^*(c)=e^{2r+2i\pi\theta}$ with $\theta\in\R/\Z$ and $r>0$. Then $\{P_\lambda\}\in\mathcal{S}_4^{\textup{vis}}$, $P_\lambda$ has simple critical points, i.e. $\Theta(P_\lambda)\in\mathsf{Cb}_0$ and $r=g_{P_\lambda}(c_0)=2g_{P_\lambda}(c_1)=2g_{P_\lambda}(c_2)$. Moreover, if $\Theta(P_\lambda)=(\Theta_0,\Theta_1,\Theta_2)$, then $\Theta(P_{-\lambda})=(\Theta_0,\Theta_2,\Theta_1)$ and
\[-2\Theta_0=4\Theta_1=4\Theta_2=\{\theta\}~.\]
\end{lemma}

\begin{proof}
As seen above, $P_\lambda$ has simple critical points if and only if $c\neq0$, which is the case here since $g_{P_\lambda}(c_0)>0$.
Notice that $P_\lambda(c_1)=P_\lambda(c_2)=f_\lambda(c_0)=c$ in our case, hence this point belongs to the ray of angle $\theta$ by assumption. Moreover, since $g_{P_\lambda}=\log|\phi_\lambda|=\log|\psi_c|$ on $\{g_{P_\lambda}>0\}$, we have $G(P_\lambda)=g_{P_\lambda}(c_0)=2g_{P_\lambda}(c_1)=2g_{P_\lambda}(c_2)$. As a consequence $g_{P_\lambda}(P_\lambda(c_0))=4g_{P_\lambda}(c_0)>G(P_\lambda)$ which means that $P_\lambda(c_0)=f_\lambda(c)$ belongs to the ray of angle $-2\theta$.

 Finally, let $\alpha$ and $\alpha+\frac{1}{2}$ be the angles so that $-2\alpha=-2(\alpha+\frac{1}{2})=\theta$. In particular, $4\alpha=4(\alpha+\frac{1}{2})=-2\theta$ and the two dynamical rays of angle $\alpha$ and $\alpha+\frac{1}{2}$ don't cross critical points of $P_\lambda$ until they terminate at $c_0$ by \cite[Lemma 3.9]{kiwi-portrait}. Since $c_1,c_2\notin\bigcup_{n\geq1}P_\lambda^{-n}\{c_0\}$, using again \cite[Lemma 3.9]{kiwi-portrait}, we have $2$ distinct rays terminating at $c_1$ (resp. at $c_2$), hence $P_\lambda\in\mathcal{S}_4^{\textup{vis}}$.
 
The last assertion follows immediately, since taking $P_{-\lambda}$ instead of $P_\lambda$ only exchanges the roles of $c_1$ and $c_2$.
\end{proof}

\subsection{Infinitely many combinatorics with non-trivial impression}

We now explain how to obtain Example~\ref{tm:non-trivial} from previous works and the above section. Remark that, when $c=(a+ib)^2\in\C$ and $\lambda:=(a,b)\in\R^2$, then both $P_\lambda$ and $P_{-\lambda}$ are the polynomial map $f_c^2$. As an immediate consequence of Lemma~\ref{lm:impressionanti}, we get the following.

\begin{corollary}\label{cor:impressionanti}
Let $c=(a+ib)^2\in\C$ with $\lambda:=(a,b)\in\R^2$, let $\theta\in\R/\Z$ and let $r>0$. Assume that $c\in\mathcal{R}^*(\theta)$ and $\Psi^*(c)=e^{2r+2i\pi\theta}$. Then $\mathcal{I}_{\mathcal{C}_4}(\Theta(P_\lambda))$ and $\mathcal{I}_{\mathcal{C}_4}(\Theta(P_{-\lambda}))$ both contain a copy of the prime end impression of the angle $\theta$ under the map $\Psi^*$.
\end{corollary}

Let us now explain how to deduce Example~\ref{tm:non-trivial} from the above.

\begin{proof}[Constructing Example~\ref{tm:non-trivial}]
Pick any $\theta\in\R/\Z$ which is $k$-periodic under multiplication by $-2$, with $k>1$ odd. According to Theorem 1.1 of~\cite{inoumuk}, the impression $\textup{Imp}(\theta)$ contain a non-trivial real-analytic arc $\mathcal{C}$ such that $f_c$ has a $k$-periodic cycle of multiplier $1$.

By Corollary~\ref{cor:impressionanti} above, this implies that both
$\mathcal{I}_{\mathcal{C}_4}(\Theta(P_\lambda))$ and $\mathcal{I}_{\mathcal{C}_4}(\Theta(P_{-\lambda}))$ contains a non-trivial analytic arc along which the maps are not affine conjugate. As a consequence, $\mathcal{I}_{\mathcal{C}_4}(\Theta(P_\lambda))$ and $\mathcal{I}_{\mathcal{C}_4}(\Theta(P_{-\lambda}))$ are non-trivial. 

It also implies that for any $(P,c_0,c_1,c_2)$ in such an arc, $P=f_c^2$ has a $k$-periodic parabolic cycle of multiplier $1$. In particular, $c_0=0$ is attracted towards this parabolic cycle under iteration of $f_c$, hence under iteration of $P$. Notice that the parabolic basin of $P$ is $f_c$-invariant. Finally, by Lemma~\ref{lm:complexTricorn2}, $f_c(c_1)=f_c(c_2)=0$. In particular, all the critical points of $P$ are attracted towards its parabolic $k$-cycle.
\end{proof}

\section{A bifurcation measure for the Tricorn}\label{sec:antiholomorphic}

We now want to define a good bifurcation measure for the Tricorn $\Mand_2^*$ and prove equidistribution properties of specific parameters towards this bifurcation measure.

\subsection{Misiurewicz combinatorics}\label{sec:misanti}

We let $\mathsf{R}_{\textup{mis}}$ be the set of angles $\theta\in\R/\Z$ such that there exists integers $n>k>1$ for which $\theta$ satisfies $(-2)^{n-1}\theta=(-2)^{k-1}\theta$ and such that $(-2)^{n-k}\theta\neq\theta$.

We first want here to prove the following.

\begin{lemma}\label{lm:finiteintersection}
Pick $n>k>0$. Then the sets $\Per(n,k)$ and $\Per^*(n,k)$ are finite.
\end{lemma}

\begin{proof}
We first prove that $\Per(n,k)$ is a finite set (hence $\Per^*(n,k)\subset\Per(n,k)$ is also finite). Pick $\lambda=(a,b)\in\R^2$ with $c=(a+ib)^2$. Assume $c\in\Per(n,k)$. According to Lemma~\ref{lm:complexTricorn2}, 
\[P_\lambda^n(c_0)=f_c^{2n}(0)=f_c^{2k}(0)=P_\lambda^{k}(c_0)\]
and again by Lemma~\ref{lm:complexTricorn2},
\begin{eqnarray*}
P_\lambda^{n+1}(c_1) & = & f_c\circ f_c^{2n}(f_c(c_1))=f_c\circ f_c^{2n}(0)\\
& = & f_c\circ f_c^{2k}(0)=f_c\circ f_c^{2k}(f_c(c_1))=P_\lambda^{k+1}(c_1).
\end{eqnarray*}
Since $P_\lambda(c_1)=P_\lambda(c_2)$, we get $P_\lambda^{n+1}(c_2)=P_\lambda^{k+1}(c_2)$. We write 
\[\Per_j(m,l):=\{(P,c_0,c_1,c_1)\in\mathcal{P}_4\, ; \ P^m(c_j)=P^l(c_j)\}.\]
The set $\Per_j(m,l)$ is an algebraic subvariety of $\mathcal{P}_4$ and $\Per_0(n,k)\cap\Per_1(n+1,k+1)\cap\Per_2(n+1,k+1)$ is contained in the compact set $\mathcal{C}_4$, hence it is finite. By Lemma~\ref{lm:complexTricorn}, the set of $\lambda\in\C^2$ with $\pi(\lambda)\in \Per_0(n,k)\cap\Per_1(n+1,k+1)\cap\Per_2(n+1,k+1)$ is thus finite. The finiteness of $\Per(n,k)$ follows directly.
\end{proof}

For $n>k\geq 1$, we also let $\mathsf{C}^*(n,k):=\{\theta\in\R/\Z\, ; \ (-2)^{n-1}\theta=(-2)^{k-1}\theta\}$. We now want to relate Misiurewicz combinatorics with Misiurewicz parameters.

\begin{lemma}\label{lm:magic}
Pick $\theta\in\mathsf{R}_\textup{mis}$ and let $n>k>1$ be minimal such that $\theta\in \mathsf{C}^*(n,k)$.
\begin{enumerate}
\item There exists a Misiurewicz parameter $c\in\partial \Mand^*_2$ such that the prime end impression of $\theta$ under $\Psi^*$ is reduced to $\{c\}$ and $c\in\Per^*(2n,2k)$,
\item Moreover, if $n-k$ is even, then $c\in\Per^*(n,k)$.
\end{enumerate} 
\end{lemma}

\begin{proof}
First, we prove 1. Let $\Theta_0:=\{\alpha,\alpha+\frac{1}{2}\}$ be such that $-2\alpha=\theta$. Let also $\{\beta_1,\beta_2,\beta_3,\beta_4\}$ be such that $4\beta_i=\theta$ for $1\leq i\leq4$, and let $\Theta_1,\Theta_2$ be such that $\Theta_1\cup\Theta_2=\{\beta_1,\beta_2,\beta_3,\beta_4\}$ and $\Theta_1\cap\Theta_2=\emptyset$ and $\Theta:=(\Theta_0,\Theta_1,\Theta_2)\in\mathsf{Cb}_0$. 

By assumption, $4^{n}\alpha=(-2)^{2n}\alpha=(-2)^{2k}\alpha=4^{k}\alpha$. Moreover, if $4^{(n-k)}\alpha\in \Theta_0$, then $(-2)^{2(n-k)}\theta=\theta$, which is exluded since $\theta\in\mathsf{R}_\textup{mis}$. As a consequence, $\alpha$ is strictly preperiodic under the map $M_4$. Similarly, $\beta_i$ is strictly $M_4$-preperiodic for all $i$, hence $\Theta\in\mathsf{Cb}_\textup{mis}$. Moreover, according to Theorem~\ref{tm:kiwipcf}, the impression $\mathcal{I}_{\mathcal{C}_4}(\Theta)$ is reduced to a singleton $\{P_\lambda\}$ where $P_\lambda$ Misiurewicz and $\Theta_i$ is a set of angles landing at the critical point $c_i$ of $P_\lambda$.
 
By Corollary~\ref{cor:impressionanti}, this implies that the prime end impression of $\theta$ is reduced to a singleton $\{c\}$. Writing $\lambda=(a,b)$, we thus have $c=(a+ib)^2$ and $f_c^{2n}(0)=P_\lambda^{n}(0)=P_\lambda^{k}(0)=f_c^{2k}(0)$, i.e. $c\in\Per(2n,2k)$. As $c$ is contained in the prime end impression of $\theta$ un $\Psi^*$, $c$ lies on the boundary of $\Mand_2^*$. If we had $f_c^{2(n-k)}(0)=0$, then $c$ would be a center of a hyperbolic component of $\Mand_2^*$ which contradicts the fact that $c\in\partial\Mand_2^*$ (see e.g. \cite{HubbardSchleicher}).

To prove $2$, if $n$ is even, we may proceed exactly as above, replacing $n$ and $k$ respectively with $n/2$ and $k/2$. Otherwise, we replace $n$ and $k$ with $(n+1)/2$ and $(k+1)/2$ respectively. This ends the proof.
\end{proof}

Notice that this result can be understood as follows: Misiurewicz combinatorics have to cluster to Misiurewicz parameters which are countable, i.e. we naturally have a rigidity property for the impression of such combinatorics.
On the other hand, Parabolic combinatorics, i.e. periodic ones under multiplication by $-2$, cluster on a set of parameters having to few rigid constraints to impose such a property in general (see Corollary~\ref{cor:impressionanti}).

\subsection{Landing of rays and the bifurcation measure}

We want first to prove the following, in the spirit of Theorem~\ref{tm:contlanding}.

\begin{theorem}\label{tm:landinganti}
There exists a set $\mathsf{R}\subset\R/\Z$ of full Lebesgue measure such that the map $\Phi:\,(\theta,r)\in\R/\Z\times\R_+^*\longmapsto \left(\Psi^*\right)^{-1}(e^{2r+2i\pi\theta}) \in \C\setminus\Mand_2^*$ extends continuously to $\mathsf{R}\times\{0\}$. Moreover, this set contains the set $\mathsf{R}_{\textup{mis}}$ and the extended map $\Phi$ induces a surjection between $\mathsf{R}_{\textup{mis}}\times\{0\}$ and the set of Misiurewicz parameters.
\end{theorem}
We follow Smirnov~\cite{Smirnov} and Dujardin and Favre~\cite{favredujardin}.
\begin{proof}
As for the proof of Theorem~\ref{tm:contlanding}, we rely on Theorem~\ref{tm:combcont}. To an angle $\theta\in\R/\Z$, owing to Lemma~\ref{lm:impressionanti}, we can associate $\Theta\in \mathsf{Cb}_0$ and we let $\mathsf{R}_1$ be the set of angles $\theta$ such that the associated $\Theta$ satisfies that the impression $\mathcal{I}_{\mathcal{C}_4}$ contains a polynomial $P_\lambda$ satisfying the TCE condition. Let $\mathsf{R}:=\mathsf{R}_1\cup\mathsf{R}_{\textup{mis}}$. According to Corollary~\ref{cor:impressionanti} and to Theorem~\ref{tm:kiwipcf}, the map $\Phi$ extends as a continuous map to $\mathsf{R}\times\{0\}$ and the extended map $\Phi$ induces a surjection from $\mathsf{R}_{\textup{mis}}\times\{0\}$ to the set of Misiurewicz parameters. Indeed, any $\theta\in\mathsf{R}_{\textup{mis}}$ corresponds exactly to $2$ distinct $\Theta,\Theta'\in\mathsf{Cb}_{\textup{mis}}$. According to Lemma~\ref{lm:impressionanti} they correspond respectively to distinct $(a,b),(a',b')\in\R^2$ with $c=(a+ib)^2$. The 
conclusion follows from Theorem~\ref{tm:kiwipcf}.

It remains to prove that $\mathsf{R}$ has full-Lebesgue measure. Notice that $\mathsf{R}_{\textup{mis}}$ is countable, hence satisfies $\dim_H(\mathsf{R}_{\textup{mis}})=0$. Hence it has Lebesgue measure $0$ and the full measure property will be fulfilled by the set $\mathsf{R}_1$. Let
\[\mathsf{Cb}^1:=\{\Theta=(\Theta_0,\Theta_1,\Theta_2)\in\mathsf{Cb}_0\, ; -2\Theta_0=4\Theta_1=4\Theta_2\}.\]
Lemma~\ref{lm:impressionanti} allows to define a map $\pi:\mathsf{Cb}^1\longrightarrow\R/\Z$. It is clear that $\pi$ is a degree $2$ unbranched cover. In particular, $\dim_H(\R/\Z\setminus\mathsf{R})=\dim_H(\mathsf{Cb}^1\setminus\pi^{-1}\mathsf{R}_1)$. Following exactly the proof of \cite[Lemma~7.25]{favredujardin}, we get 
\begin{claim}
$\dim_H(\R/\Z\setminus\mathsf{R})\leq\log3/\log4<1.$
\end{claim}
In particular, $\lambda_{\R/\Z}(\R/\Z\setminus\mathsf{R})=0$ which ends the proof.
\end{proof}

\begin{proof}[Proof of The Claim]
Pick $\Theta\in\mathsf{Cb}_1$ and let $c=(a+ib)^2$ be such that $P_{a,b}\in\mathcal{I}_{\mathcal{C}_4}(\Theta)$. According to \cite[Lemma 7.24]{favredujardin}, there exists a partition of $\J_{a,b}$ into three sets $J_0$, $J_1$ and the impression of the external ray of $P_{c,a}$ in $\J_{a,b}$. Let $\Sigma_2:=\{0,1\}^\mathbb{N}$ and $\kappa:(\R\setminus\Q)/\Z\longrightarrow\Sigma_2$ be defined as follows: we say that $\kappa(\theta)_n=\epsilon\in\{0,1\}$ if $4^n\theta\in I_\epsilon$, where $I_\epsilon$ is the connected component of $\R/\Z\setminus\{\alpha,\alpha+\frac{1}{2}\}$, with $-2\alpha=\theta$, such that angles in $I_\epsilon$ land in $J_\epsilon$.

Following Smirnov~\cite{Smirnov}, we see that $\theta\in\R/\Z$ fails the TCE condition if and only if $\theta\in\kappa^{-1}(SR)$ i.e. is strongly recurrent. The precise definition of $SR$ can be found in~\cite{Smirnov}. It is known that $SR$ has Hausdorff dimension $0$ and, following Dujardin and Favre, if $C_n$ is any cylinder of depth $n$ in $\Sigma_2$, $\kappa^{-1}(C_n)$ consists of the intersection of $(\R\setminus\Q)/\Z$ with at most $A \times n 3^n$ intervals of length at most $2^{-n}$, where $A$ is a constant independent of $n$.

Indeed, if $4^n\theta$ turns once around $\R/\Z$ so that $4^n\theta=\theta$, then $\theta\in\Q$, which is excluded. We now proceed by induction on $n$: let $I$ be an interval of $\R/\Z$ of length $\ell_I<1/4$. Then, either $4\theta\notin I$ and $M_d^{-1}(I)\cap I_\epsilon$ consists in $2$ intervals, or $4\theta\in I$ and $M_d^{-1}(I)\cap I_\epsilon$ consists in $3$ intervals, which can occur only for one of the $N(n)$ intervals. As a consequence, $N(n+1)\leq 2\cdot (N(n)-1)+3\leq 2 N(n)+1$, whence $N(n)\leq A \times n 3^n$. The estimate for the Hausdorff dimension then easily follows.
\end{proof}

Thanks to Theorem~\ref{tm:landinganti}, we can define the landing map and the bifurcation measure for the Tricorn.

\begin{definition}
 The \emph{landing map of the Tricorn} is the measurable map $\ell:\R/\Z\rightarrow\partial\Mand_2^*$ defined by $\ell(\theta):=\Phi(e^{2i\pi\theta},0)$ for any $\theta\in\mathsf{R}$. We define the \emph{bifurcation measure of the Tricorn} as the probability measure
\[\mu_\bif^*:=\ell_*\left(\lambda_{\R/\Z}\right).\]
 \end{definition} 

As an immediate consequence of Theorem~\ref{tm:landinganti}, we have proven Theorem~\ref{tm:mubifanti}. Notice also that we have $\mathsf{R}_{\textup{mis}}\subset\mathsf{R}$.

\subsection{Distribution of Misiurewicz combinatorics: Theorem~\ref{tm:distribanti}}

Recall that, for $n>k>1$, we denoted by $\Per^*(n,k)$ the set of parameters such that $f_c^n(0)=f^{k}_c(0)$ and $f^{n-k}_c(0)\neq0$.
Observe that we do not consider the case $\Per^*(n,1)$ since that set is empty. Indeed, since the map $f_c$ has local degree $2$ at $0$, the point $f_c(0)$ cannot have a preimage distinct from $0$. In particular, any parameters for which $f_c^n(0)=f_c(0)$ satisfies $f^{n-1}_c(0)=0$.

For any $n\geq4$, pick $1<k(n)<n$. Let $\mathsf{X}_n:=\mathsf{C}^*(n,k(k))\setminus\mathsf{C}^*(n-k(n)+1,1)$, $d_n:=\textup{Card}(\mathsf{X}_n)=2^{n-1}-2^{k(n)-1}-2^{n-k(n)}+1$ and
\[\nu_n:=\frac{1}{d_n}\sum_{\theta\in\mathsf{X}_n}\delta_\theta.\]
Let also $X_n:=\ell\left(\mathsf{X}_n\right)\subset\Per^*(2n,2k(n))$, and let $\mu_n^*$ be the probability measure 
\[\mu_n^*:=\frac{1}{d_n}\sum_{c\in X_n}\mathcal{N}_{\R/\Z}(c)\cdot\delta_c~,\]
where $\mathcal{N}_{\R/\Z}(c)\geq1$ is the number of angles $\theta\in\mathsf{R}_{\textup{mis}}$ for which $\ell(\theta)=c$, i.e. $\mu_n^*=\ell_*(\nu_n)$.

Our aim here is to prove the following.

\begin{theorem}\label{tm:distribantiangle}
The sequence of measures $(\nu_n)_{n\geq4}$ converges weakly to $\lambda_{\R/\Z}$.
\end{theorem}

\begin{proof}
As in the proof of the above lemma, if $n-k(n)$ is even, we have 
\begin{eqnarray*}
d_n & = & 2^{n-1}-2^{k(n)-1}-2^{n-k(n)}+1\\
& \geq & 2^{n-1}-2^{k(n)-1}-\frac{1}{2}\left(2^{n-1}-2^{k(n)-1}\right)\\
& \geq & \frac{1}{2}\textup{Card}(\mathsf{C}^*(n,k(n)))~,
\end{eqnarray*}
and if $n-k(n)$ is odd, we also have
\begin{eqnarray*}
d_n & = & 2^{n-1}+2^{k(n)-1}-2^{n-k(n)}-1\\
& \geq & 2^{n-1}+2^{k(n)-1}-\frac{1}{2}\left(2^{n-1}+2^{k(n)-1}\right)\\
& \geq & \frac{1}{2}\textup{Card}(\mathsf{C}^*(n,k(n)))~.
\end{eqnarray*}
Moreover, it is easy to see that the sequence $\left(\mathsf{C}^*(n,k(n))\right)_{n\geq2}$ is equidistributed for $\lambda_{\R/\Z}$ and that $\left(\mathsf{C}^*(n-k(n)+1,1)\right)_{n\geq2}$ is either equidistributed or $k(n)\geq n-K$ for some integer $K\geq1$. By Lemma~\ref{lm:equidsitrcircle}, the sequence $(\mathsf{X}_n)_n$ is equidistributed. In particular, the sequence of probability measures $(\nu_n)$ converges weakly towards $\lambda_{\R/\Z}$.
\end{proof}

We now can prove Theorem~\ref{tm:distribanti}.

\begin{proof}[Proof of Theorem~\ref{tm:distribanti}]
It follows directly from Theorem~\ref{tm:landinganti}, Lemmas~\ref{lm:finiteintersection} and \ref{lm:magic} combined with Theorem~\ref{tm:distribantiangle} and Theorem~\ref{tm:measure}.
\end{proof}

\begin{remark}
\textup{We expect the measure which is \emph{equidistributed} on the set $\Per^*(n,k(n))$ converges towards to the bifurcation measure $\mu_\bif^*$, whenever $1< k(n)<n$.}
\end{remark}

\section{The bifurcation measure of the Tricorn and real slices}
Let $\mu_{\bif,d}$ (resp. $\nu_{\bif,d}$) denote the bifurcation measure on the moduli space $\mathcal{P}_d$ of degree $d$ polynomials with marked critical points (resp. on the moduli space  $\mathcal{M}_d$ of degree $d$ rational maps with marked critical points). Let $\{P\}\in\mathcal{P}_d$ (resp. $\{R\}\in\mathcal{M}_d$), then $\{P^{\circ m}\} \in\mathcal{P}_{d^m}$ (resp. $\{R^{\circ m}\} \in\mathcal{M}_{d^m}$). Let $\mathrm{it}_m:\{P\}\mapsto \{P^{\circ m}\}$ (resp. $\mathrm{it}_m:\{R\}\mapsto \{R^{\circ m}\}$).
We have the following interesting result that relates the image of the support of the bifurcation measure under the iteration map $\mathrm{it}_m$ with the support of the bifurcation measure. It is new to the authors' knowledge. More precisely:
\begin{proposition}\label{prop_inclusion_slice}
Let $m\geq 2$, then with the above notations, we have that:
\begin{itemize}
	\item if $\{P\}\in\mathcal{P}_d$ satisfies $\{P\}\in \mathrm{supp}(\mu_{\bif,d})$ then $\{P^{\circ m}\}\in \mathrm{supp}(\mu_{\bif,d^m})$. In other words 
	\[ \mathrm{it}_m(\mathrm{supp}(\mu_{\bif,d})) \subset \mathrm{supp}(\mu_{\bif,d^m}) \cap \mathrm{it}_m(\mathcal{P}_d);\]
	\item if $\{R\}\in\mathcal{M}_d$ satisfies $\{R\}\in \mathrm{supp}(\nu_{\bif,d})$ then $\{R^{\circ m}\}\in \mathrm{supp}(\nu_{\bif,d^m})$. In other words 
	\[ \mathrm{it}_m(\mathrm{supp}(\nu_{\bif,d})) \subset \mathrm{supp}(\nu_{\bif,d^m}) \cap \mathrm{it}_m(\mathcal{M}_d).\]
\end{itemize}
\end{proposition}
\begin{proof}
We first prove the case where $\{P\}\in\mathcal{P}_d$. By density of Misiurewicz parameters in $\mathrm{supp}(\mu_{\bif,d})$ (see \cite{favredujardin}), we can assume that $\{P\}$ is Misiurewicz, i.e. all critical points of $P$ are preperiodic to repelling cycles. Then $P^{\circ m}$ is also Misiurewicz since its critical points are the preimages of the critical points of $P$ by $P^{\circ k}$ for $k\leq m-1$. But then, it is known that such conjugacy classes $\{P^{\circ m}\}$ belong to $\mathrm{supp}(\mu_{\bif,d^m})$ (see again \cite{favredujardin}). This ends the proof in $\mathcal{P}_d$.

The case of rational maps is similar (see \cite{buffepstein}, \cite{Article1} and \cite{Article2}).
\end{proof}

\begin{question} \rm Does the reverse inclusion hold? Namely, if $\{P^{\circ m} \}\in \mathrm{supp}(\mu_{\bif,d^m})$, is it true that $\{P\} \in \mathrm{supp}(\mu_{\bif,d})$ ? We expect the answer to be positive.
\end{question}

We proceed similarly for $\mu_\bif^*$. Let $\pi:\C\rightarrow\mathcal{P}_4$ be the map defined by 
\[\pi(c):=([f_{c^2}^2],0,c_1,c_2), \ c\in\C,\]
where $(c_1,c_2):=(i\bar{c},-i\bar{c})$.
Let also $\mu$ be the bifurcation measure of the complex surface $\mathcal{X}:=\{c_1=-c_2\}$ of the moduli space $\mathcal{P}_4$ of degree $4$ critically marked polynomials and let $\mathcal{S}\subset\mathcal{X}$ be the smooth real surface defined as the image of the map $\pi$. Then proceeding as in the proof of Proposition~\ref{prop_inclusion_slice}, we have:
\begin{proposition}\label{tm:mubifantislice}
Let $c \in \C$ be such that $c \in \mathrm{supp}(\mu_\bif^*)$, then $\pi(c) \in \mathrm{supp}(\mu)$. In other words, $\pi(\mathrm{supp}(\mu_\bif^*)) \subset \mathrm{supp}(\mu)\cap \mathcal{S}$.
\end{proposition}
The fact that Misiurewicz parameters belong to the support of $\mu$ follows from \cite{Article1}.

\begin{question} \rm Is $\mu_\bif^*$ the \emph{slice} $\langle\mu,\mathcal{S}\rangle$ in the sense of measures of $\mu$ along $\mathcal{S}$ ?
It was one of the initial approach we had to construct $\mu_\bif^*$. Notice that it is not clear that $\langle\mu,\mathcal{S}\rangle$ is even well-defined (see for instance \cite{tiozzo} for the delicate question of the real slicing of the harmonic measure of the Mandelbrot set). Numerical evidences of such results, in the spirit of Milnor's explorations (\cite{Milnor5}), would be a first step.
\end{question}

\bibliographystyle{short}
\bibliography{biblio}

\end{document}